\pdfoutput=1
\documentclass[11pt]{amsart}

\usepackage{amssymb, amsmath, amsthm, amsfonts}
\usepackage{mathtools}
\usepackage{xcolor}

\usepackage{enumerate}
\usepackage[unicode]{hyperref}
\hypersetup{
bookmarks=true,
colorlinks=true,
citecolor=[rgb]{0,0,0.5},
linkcolor=[rgb]{0,0,0.5},
urlcolor=[rgb]{0,0,0.75},
pdfpagemode=UseNone,
pdfstartview=FitH,
pdfdisplaydoctitle=true,
pdflang=en-US
}

\usepackage[left=1.03in, right=1.03in, top=1.05in, bottom=1.05in]{geometry}

\title[A stationary set method for oscillatory integrals]{A stationary set method\\ for estimating oscillatory integrals}
\author{Saugata Basu}
\address[SB]{Department of Mathematics\\ Purdue University\\ USA}
\email{sbasu@math.purdue.edu}
\author{Shaoming Guo}
\address[SG]{Department of Mathematics\\ University of Wisconsin Madison\\ USA}
\email{shaomingguo@math.wisc.edu}
\author{Ruixiang Zhang}
\address[RZ]{Department of Mathematics, University of California, Berkeley\\ USA}
\email{ruixiang@berkeley.edu}
\author{Pavel Zorin-Kranich}
\address[PZ]{Mathematical Institute\\ University of Bonn\\ Germany}
\email{pzorin@math.uni-bonn.de}

\def\R{\mathbb{R}}

\def\N{\mathbb{N}}
\def\C{\mathbb{C}}

\newcommand{\Part}[1]{\mathcal{P}(#1)}

\newcommand{\con}{\mathrm{Con}}
\newcommand{\dual}{\mathrm{Dual}}
\newcommand{\bxi}{\bar{\xi}}
\newcommand{\bdelta}{{\bf \Delta}}

\newcommand{\one}{\mathbf{1}}

\newcommand{\bfr}{{\bf r}}
\newcommand{\an}{\mathrm{an}}

\def\mc{\mathcal}
\def\lesim{\lesssim}

\def\PZdefchar#1{
	\expandafter\def\csname frak#1\endcsname{\mathfrak{#1}}
	\expandafter\def\csname bf#1\endcsname{\mathbf{#1}}
	\expandafter\def\csname scr#1\endcsname{\mathcal{#1}}
	\expandafter\def\csname cal#1\endcsname{\mathcal{#1}}}
\def\PZdefloop#1{\ifx#1\PZdefloop\else\PZdefchar#1\expandafter\PZdefloop\fi}
\PZdefloop abcdefghijklmnopqrstuvwxyzABCDEFGHIJKLMNOPQRSTUVWXYZ\PZdefloop

\DeclarePairedDelimiter\abs{\lvert}{\rvert}

\DeclarePairedDelimiter\norm{\lVert}{\rVert}

\DeclarePairedDelimiterX\inn[2]{\langle}{\rangle}{#1,#2}
\providecommand\given{}
\newcommand\SetSymbol[1][]{%
	\nonscript\:#1\vert
	\allowbreak
	\nonscript\:
	\mathopen{}}
\DeclarePairedDelimiterX\Set[1]\lbrace\rbrace{\renewcommand\given{\SetSymbol[\delimsize]}#1}

\newcommand{\extend}{E_k 1}
\newcommand{\gammax}{Z_x}
\newcommand{\qk}{q_k}

\numberwithin{equation}{section}
\theoremstyle{plain}
\newtheorem{thm}{Theorem}[section]
\newtheorem{prop}[thm]{Proposition}

\newtheorem{lem}[thm]{Lemma}
\newtheorem{lemma}[thm]{Lemma}
\newtheorem{cor}[thm]{Corollary}
\newtheorem{corollary}[thm]{Corollary}
\theoremstyle{definition}

\newtheorem*{conj*}{Conjecture}

\newtheorem*{openproblem*}{Open Problem}
\theoremstyle{remark}
\newtheorem{rem}[thm]{Remark}
\newtheorem{remark}[thm]{Remark}
\newtheorem{exam}[thm]{Example}

\newcommand {\hide}[1]{}

\begin{document}

\maketitle
\begin{abstract}
We propose a new  method of estimating oscillatory integrals, which we call a ``stationary set'' method.
We use it to obtain the sharp convergence exponents of Tarry's problems in dimension two for every degree $k\ge 2$.
As a consequence, we obtain sharp $L^{\infty} \to L^p$ Fourier extension estimates for a family of monomial surfaces.
\end{abstract}

\section{Statement of the main results}

Our main result, Theorem \ref{201004lem2_1_general}, gives upper bounds on certain oscillatory integrals of bounded complexity.
\begin{thm}[``Stationary set'' estimate]\label{201004lem2_1_general}
Let $d, k\ge 1$. Let $Q(\xi)$ be a bounded semialgebraic function in variables $\xi\in [0, 1]^d$ with complexity $\leq k$.
Then
\begin{equation}
\label{stationary_set_bound}
\Big|\int_{[0, 1]^d} e(Q(\xi))d\xi\Big|\lesim_{d, k} \sup_{\mu\in\R} \abs{\{\xi\in [0, 1]^d: \mu\le Q(\xi)\le \mu+1\}}.
\end{equation}
The implicit constant depends only on $d$ and $k$.
\end{thm}

Theorem~\ref{201004lem2_1_general} relies mainly on tools from real algebraic geometry and model theory, in particular, o-minimal geometry.

We will apply Theorem \ref{201004lem2_1_general} to obtain new and often sharp results on the convergence exponent of Tarry's problem  in dimension $2$ and $L^{\infty} \to L^p$ Fourier restriction for two-dimensional Parsell-Vinogradov surfaces.
These applications only require Theorem~\ref{201004lem2_1_general} with functions $Q$ being polynomials of bounded degrees, but our proof does not simplify in this special case.
The more general setting of semialgebraic functions may also be of independent interest.

\subsection{Sharpness of Theorem \ref{201004lem2_1_general} in an average sense.} The argument in this subsection was shared to us by Wright, who further attributed it to Stein in the real setting and to Hua (see \cite[page 15]{MR0194404}) in the $p$-adic setting. A similar form of it was used for instance in \cite{wright2020h}. Their argument also works for general valued fields, Archimedean or non-Archimedian. Here we only present this argument in the real setting. It shows that the upper bound \eqref{stationary_set_bound} is tight in an average sense (up to a log factor). Let us be more precise. Denote
\begin{equation}
L_Q:=\sup_{c\in \R} |\{\xi\in [0, 1]^d: c\le Q(\xi)\le c+1\}|
\end{equation}
and
\begin{equation}
I(s Q)=\int_{[0, 1]^d}e(sQ(\xi))d\xi.
\end{equation}
We will prove that
\begin{equation}\label{210828e1_3zz}
\int_{-1}^1 |I(sQ)|ds\lesim \int_{-1}^1 L_{sQ} ds\lesim (|\log \delta|+1) \int_{-1}^1 |I(sQ)|ds,
\end{equation}
where $Q$ is a bounded semialgebraic function of complexity $\le k$, $\delta:=L_Q\le 1$, and the implicit constants are allowed to depend on $d$ and $k$. \\

The first inequality in \eqref{210828e1_3zz} follows directly from Theorem \ref{201004lem2_1_general}. We only need to prove the second inequality. Let $g: [0, 1]^d\to \R$ be a measurable function. Let $\psi: \R\to \R$ be a Schwartz function satisfying $\mathbf{1}_{[0, 1]}(x)\le \psi(x)$ for every $x\in \R$ whose Fourier transform is compactly supported. 
We first show that
\begin{equation}\label{210828e1_3}
L_g\le \int_{\R} |\widehat{\psi}(s)| |I(sg)|ds.
\end{equation}
To see this, we write
\begin{equation}
|\{\xi\in [0, 1]^d: c\le g(\xi)\le c+1\}|=\int_{[0, 1]^d} \mathbf{1}_{[0, 1]}(g(\xi)-c)d\xi.
\end{equation}
By the definition of $\psi$, this is bounded by
\begin{equation}
\begin{split}
& \int_{[0, 1]^d} \psi(g(\xi)-c)d\xi=\int_{[0, 1]^d} \big[\int_{\R} \widehat{\psi}(s)e(s(g(\xi)-c))ds\big] d\xi
\end{split}
\end{equation}
The bound \eqref{210828e1_3} now follows from the triangle inequality. \\

%

Recall that $\delta=L_Q$. Assume that $\widehat{\psi}(s)$ is supported on $[-C, C]$. For every $j$ with $\delta\lesim 2^{-j}\le 1$, it holds that
\begin{equation}
\int_{|s|\sim 2^{-j}}  L_{sQ} ds \sim 2^{-j} L_{2^{-j}Q} \sim 2^{-j} L_{2^{-j}Q/C} \lesim 2^{-j} \int_{-C}^{C} |I(2^{-j}s Q/C)|ds \lesim \int_{-1}^{1} |I(sQ)|ds,
\end{equation}
where in the second last inequality we applied \eqref{210828e1_3}. This, combined with the trivial bound
\begin{equation}
\int_{|s|\le \delta} L_{sQ} ds\le \int_{|s|\le \delta} 1 ds\le 2\delta \lesim \int_{-1}^{1} |I(sQ)|ds,
\end{equation}
implies the second inequality in \eqref{210828e1_3zz}.

\subsection{Comparison to previously known tools for estimating oscillatory integrals}

A well-known method of estimating oscillatory integrals is through van der Corput estimates, which assume a lower bound for $D^{\beta} Q$ for some multi-index $\beta\in \N^d$.
This method is very efficient in dimension one, see for instance Stein \cite{MR1232192}, or Arkhipov, Chubarikov and Karatsuba \cite[Section 1.1]{MR552548} for a more quantitative, but perhaps lesser-known, formulation.
In \cite{MR1683156}, Carbery, Christ, and Wright established van der Corput estimates in dimensions greater than one with very mild assumptions on the relevant phase functions. We also refer to Wright \cite{MR2755683}, \cite{MR2784669} and \cite{wright2020h} for estimating multi-dimensional oscillatory integrals in settings other than $\R^d$ (for instance the $p$-adic setting).

The results in the literature that are perhaps closest to ours were obtained by Bruna, Nagel and Wainger \cite{MR932301} (see also Nagel, Seeger and Wainger \cite{MR1231151} for a related result).
To describe this result, let $S\subset \R^{n+1}$ be a smooth, compact hypersurface that is convex and of finite type.
Let $d\sigma$ be the surface measure on $S$.
For $x\in S$, let $\nu_x$ denote the outward unit normal to $S$ at $x$.
Bruna, Nagel, and Wainger \cite[Theorem B]{MR932301} proved that, for large $\lambda\gg 1$, we have
\begin{equation}
\label{eq:BNW}
\Bigl| \int_S e(\inn{y}{\lambda\nu_{x}}) d\sigma(y) \Bigr|
\lesim
\sigma \{y\in S: |\inn{y-x}{\lambda\nu_{x}}|<1 \}.
\end{equation}
The left-hand side of \eqref{eq:BNW} is the Fourier transform of the surface-carried measure $s$.
The neighborhood of $x$ on the right-hand side of \eqref{eq:BNW} plays a similar role to our stationary sets in Theorem~\ref{201004lem2_1_general}.

In a slightly different direction, Varchenko \cite{MR0422257} obtained asymptotic estimates for oscillatory integrals
\begin{equation}
\int_{\R^d} e(\lambda \varphi(x))\chi(x)dx,
\end{equation}
where $\lambda\in \R$ is a large parameter, $\varphi$ is an analytic phase function, and $\chi$ is a function supported near one point, say the origin. The asymptotic expansion involves terms in the Newton polyhedron of the phase function.  Instead of asking a ``global'' condition of the form $D^{\beta} \varphi\ge 1$ as in \cite{MR1683156}, Varchenko assumed the support of $\chi$ to be small enough; he used the resolution of singularities for analytic functions, and his results are more ``local''.
For further work along this line, we refer to Karpushkin \cite{MR585530} and \cite{MR731895}, Phong, Stein and Sturm \cite{MR1738409} and Greenblatt \cite{MR2587095}, and the references therein.
Operator-valued oscillatory integrals in a similar spirit have also been extensively studied. It is beyond our capability to give a review of even a small part of it, and therefore we refer interested readers to more classical works including Phong and Stein \cite{MR1484770} and Seeger, Sogge and Stein \cite{MR1127475}, and more recent works including Ikromov, Kempe and M\"uller \cite{MR2653054} and Ikromov and M\"uller \cite{MR3524103}.

\subsection{Applications of the stationary set estimate}

As an application of our stationary set estimate, we obtain the sharp convergence exponent of Tarry's problem in dimension two for an arbitrary degree. Let $k\ge 2$ and $\xi=(\xi_1, \dots, \xi_d)\in \R^d$ with $d\ge 1$. Set
\begin{equation}
N=N_{d, k}=\binom{d+k}{k}-1.
\end{equation}
For a multi-index $\beta=(\beta_1, \dots, \beta_d)\in \N_0^d$, denote $\xi^{\beta}=\xi_1^{\beta_1}\dots \xi_d^{\beta_d}$. Let $\phi_{d, k}(\xi)$ denote the vector $(\xi^{\beta})_{1\le |\beta|\le k}\in \R^N$.
Let $S_{d, k}$ denote the $d$-dimensional manifold
\begin{equation}
\{\phi_{d, k}(\xi): \xi\in [0, 1]^d\},
\end{equation}
which is sometimes referred to as a \emph{Parsell-Vinogradov manifold}.
For $W\subset [0, 1]^d$, define the \emph{Fourier extension operator} associated to $W$ to be
\begin{equation}\label{sharp_cutoff}
E^{(d, k)}_W f(x)=\int_W f(\xi) e(x\cdot \phi_{d, k}(\xi))d\xi, \ \ x\in \R^N.
\end{equation}
The convergence exponent of Tarry's problem is defined to be
\begin{equation}
\label{eq:convergence-exponent}
p_{d, k}:=\inf \{p: \|E_{[0, 1]^d}^{(d, k)} 1\|_{L^p(\R^{N})}<\infty\}.
\end{equation}
Theorem \ref{201004lem2_1_general} will serve as a main ingredient in the proof of the following theorem, in which we obtain the sharp convergence exponent of Tarry's problem in dimension two.
\begin{thm}\label{main_tarry}
For every $k\ge 2$, we have
\begin{equation}
p_{2, k}=\frac{1}{6} k(k+1)(k+2)+2.
\end{equation}
\end{thm}
The quantity $\|E_{[0, 1]}^{(1, k)} 1\|_{L^p(\R^{N})}$ appears in the leading coefficient of the asymptotic expansion of Vinogradov's mean value, as was first shown by Hua for large exponents $p$ \cite[Chapter 10.3]{MR0029416}, and in the full range of $p$ in \cite[Corollary 1.3]{MR3938716}.
The problem of determining the convergence exponent \eqref{eq:convergence-exponent} goes back to Hua's works \cite{MR0064806,Hua}.
In our notation, Hua proved that $p_{1, k}\le k^2/2+k$, and posed the problem of determining $p_{1,k}$.
This problem was resolved by Arkhipov, Chubarikov and Karatsuba \cite{MR552548} (see also \cite[Theorem 1.3]{MR2113479} in their book), where they proved that
\begin{equation}\label{eq:one_tarry}
p_{1, k}=(k^2+k)/2+1
\end{equation}
for every $k$.
For the connection to the original Tarry problem, we refer to \cite[Section 13]{MR3938716}.


In the same paper \cite{MR552548}, the authors also studied Tarry's problem in higher dimensions, and obtained that
\begin{equation}\label{201219e1_9}
p_{d, k}\le d\binom{k+d}{d+1}+d,
\end{equation}
for every $d, k\ge 2$.
At the end of their paper, they also stated as an open problem to find the exact value of $p_{d, k}$ in higher dimensions.
Later, Ikromov \cite{MR1636721} obtained the lower bound
\begin{equation}
p_{d, k}\ge \binom{k+d}{d+1}+1.
\end{equation}
In the quadratic case, that is, the case $k=2$, the problem was resolved by Mockenhaupt \cite{Mockenhaupt}, who proved that $p_{d, 2}=2(d+1)$ for every $d\ge 2$.

For more partial results and for results on other systems of monomials, we refer to Chakhkiev \cite{MR1972998}, Ikromov and Safarov \cite{MR3496236}, Safarov \cite{MR3916135}, and Dzhabbarov \cite{MR3920413,MR4092565}.
\begin{remark}
In \cite{MR4092565}, Dzhabbarov considered the case $d=2, k=3$ and claimed that $p_{2, 3}\le 12.$ The strategy in that paper is to show that $E_{[0, 1]^2}^{(2, 3)}1\in L^{12}(\R^9)$.  To bound the $L^{12}$ norm, the author there expanded the 12-th power of $E_{[0, 1]^2}^{(2, 3)}1$ and applied some delicate change of variables, which further required some careful computations of Jacobians.
Although this approach looks very natural, our results below (see Section \ref{210216section5}, in particular, Remark \ref{210216rem}) indicates that $E_{[0, 1]^2}^{(2, 3)}1$ narrowly misses being $L^{12}$ integrable. Indeed,
\begin{equation}
\int_{B_R}|E_{[0, 1]^2}^{(2, 3)}1|^{12} \gtrsim \log R,
\end{equation}
for every large $R$ where  $B_R\subset \R^9$ is the ball of radius $R$ centered at the origin.
\end{remark}

As a consequence of Theorem \ref{main_tarry}, we obtain the following Fourier restriction estimates.
\begin{cor}\label{coro1_3}
Let $d=2$. For $k\ge 2$, let $\bar{p}_{d, k}$ denote the smallest even number that is greater than or equal to $p_{d, k}$. Then
\begin{equation}\label{201219e1_11}
\|E_{[0, 1]^d}^{(d, k)} f\|_p \lesim_{k, p} \|f\|_{\infty},
\end{equation}
for every $p>\bar{p}_{d, k}$. Moreover, when $p_{d, k}$ is even, which is the same as saying $k\not\equiv 1 \mod 4$, the above estimate is sharp in the sense that it fails for every $p\le \bar{p}_{d, k}$.
\end{cor}

When $k\equiv 3 \mod 4$, it is reasonable to believe that \eqref{201219e1_11} also holds for every $p>p_{2, k}$. Here our result covers the range $p>p_{2, k}+1$.

Recall that Theorem \ref{main_tarry} states that \eqref{201219e1_11} holds with $f=1$ whenever $p>p_{2, k}$. That such an estimate implies \eqref{201219e1_11} for a general $f\in L^{\infty}$ for an even $p$ is standard, and is exactly how Mockenhaupt \cite{Mockenhaupt} proved \eqref{201219e1_11} for every $p>p_{d, 2}$ and every $d\ge 2$. In the case of Mockenhaupt, the exponent $p_{d, 2}=2(d+1)$ is always even, which allows him to obtain sharp estimates in $p$. In our case, our exponent $p_{2, k}$ may be odd, and by doing this reduction we may be off the sharp exponent of $p$ by 1.

Because of its applications to areas like partial differential equations, analytic number theory, and combinatorics, the Fourier restriction theory has received significant amount of attentions in the past few decades.
A number of methods, including the bilinear method (see for instance Tao, Vargas and Vega \cite{MR1625056}, Wolff \cite{MR1836285} and Tao \cite{MR2033842}), the multilinear method (see for instance Bennett, Carbery and Tao \cite{MR2275834} and Guth \cite{MR2746348}), the broad-narrow analysis of Bourgain and Guth \cite{MR2860188} and the polynomial method (see for instance Guth \cite{MR3454378}), have been introduced to study the Fourier restriction phenomenon. Perhaps it is not exaggerating to say that each of the above mentioned papers generated an area of active research. Our method in the current paper seems to be disjoint from the above mentioned ones, and it would be very interesting to see how they can be connected.

\subsection{A few remarks on Theorem \ref{main_tarry} and Corollary \ref{coro1_3}}

Let us first remark on the methods of proof of Theorem \ref{main_tarry}, and compare with those in the literature.

First of all, in the quadratic case $k=2$ in \cite{Mockenhaupt}, in order to bound the $L^p$ norm of $E_{[0, 1]^d}^{(d, 2)} 1$, Mockenhaupt computed the function $E_{[0, 1]^d}^{(d, 2)} 1$ directly. To be more precise, he first smoothed out the function $\one_{[0, 1]^d}$ to $e^{-|\xi|^2/2}$, took the advantage of having a quadratic phase function and computed a Fourier transform directly (roughly speaking via completing squares). It seems difficult to generalize this approach to surfaces of cubic degree or higher.
The result of Mockenhaupt \cite{Mockenhaupt} has recently been recovered in \cite{guo2020decoupling} via the broad-narrow analysis of Bourgain and Guth and decoupling inequalities for quadratic forms, and the approach there does not work either for degrees $k\ge 3$.

Next, we turn to the upper bound \eqref{201219e1_9} established in \cite{MR552548}.
The authors of \cite{MR552548} first derived an accurate pointwise bound for a $d$-dimensional oscillatory integral, and then applied this bound to obtain a good pointwise bound for $|E_{[0, 1]^d}^{(d, k)} 1(x)|$ in terms of $x$.
Their pointwise bound is sharp for some, but not all, $x\in \R^N$, resulting in an integrability exponent that is not sharp.

In our paper, we propose an entirely different approach, which we call a stationary set method. This method, combined with a careful study of the geometry of stationary sets in our problem (see Subsection \ref{subsection4_1}), and certain rigidity properties of stationary sets (see Subsection \ref{201118subsection7_3}), allows us to prove sharp $L^p$ integribility of $|E_{[0, 1]^d}^{(d, k)} 1(x)|$ for $d=2$ and every $k\ge 2$.
\begin{remark}\label{newpaperannounce}
We expect our stationary set method to be useful in obtaining the sharp convergence exponents of Fourier extension operators for more general semialgebraic sets.
In a follow up paper, we are planning to show that the stationary set approach is often ``tight'' in the following sense.
Let $\phi : \R^{d} \to \R^{N}$ be a semialgebraic function and
\[
L(x):=\sup_{\mu\in \R} |\{\xi\in [0, 1]^d: \mu\le x\cdot \phi(\xi)\le \mu+1\}|,
\quad
x \in \R^{N},
\]
be the size of the stationary set associated to $x \cdot \phi$.
Define the convergence exponent and the stationary exponent by
\begin{align*}
p_c &:= \inf \Big\{p: \Big\|\int_{[0, 1]^d} e(x\cdot \phi(\xi))d\xi\Big\|_{L^p_x(\R^N)}<\infty\Big\}, \\
p_s &:=\inf \Big\{p: \|L(x)\|_{L^p_x(\R^N)}<\infty \Big\}.
\end{align*}
We will show that $p_s = \max \{p_c, N\}$.
The same principle also applies to Fourier transforms of surface measures of semialgebraic sets.

Moreover, via the method of the current paper, we will also show that, for a semialgebraic set $S\subset \R^N$ of dimension $d$, it holds that $p_c<\infty$ if and only if the intersection of $S$ with every hyperplane of $\R^N$ has dimension $<d$.

Although the problem of finding $p_c$ can be reduced to finding $p_s$ whenever $p_c>N$, we remark that already for three dimensional Parsell-Vinogradov surfaces (for which the condition $p_c>N$ holds), we encountered significant technical difficulties when estimating the relevant stationary sets.
We refer to Remark~\ref{rem:higher-dim-tarry} for a more detailed discussion.
\end{remark}

In the end, we comment on Corollary \ref{coro1_3}.
The estimate \eqref{201219e1_11} with a sharp range of $p$ when $d=1$ was proven by Drury \cite{MR781781}; indeed, he proved a stronger result with $\|f\|_{\infty}$ replaced by $\|f\|_q$ for an optimal range of $q$ as well.
As a consequence, Drury \cite{MR781781} recovered the result of Arkhipov, Chubarikov and Karatsuba \cite{MR552548} and gave an alternative proof for \eqref{eq:one_tarry}.
It is worth mentioning that Drury's result can be partially recovered via the broad-narrow analysis of Bourgain and Guth \cite{MR2860188} and the multi-linear restriction estimates of Bennett, Carbery and Tao \cite{MR2275834}; to be more precise, via these more recent tools, one can prove \eqref{201219e1_11} for an optimal range of $p$, and therefore obtain another proof of \eqref{eq:one_tarry}.  Moreover, as was shown in \cite{guo2020decoupling}, if one combines the broad-narrow analysis of Bourgain and Guth, with the multi-linear Fourier restriction estimates of Bennett, Bez, Flock and Lee \cite{MR3783217} and the decoupling inequalities for quadratic Parsell-Vinogradov surfaces in Guo and Zhang \cite{MR3994585}, then one can also recover the result of Mockenhaupt \cite{Mockenhaupt} (i.e. the case where $d\ge 2$ and $k=2$).
We tried to prove Corollary \ref{coro1_3} by this approach, but did not succeed.
For $d=2$ and $k=3$, we encountered problems in the narrow part of the analysis; while for $k\ge 4$, we encountered significant difficulties in both the broad and the narrow part.\\

\noindent \textbf{Acknowledgements.}
S.B. was supported in part by the NSF grant CCF-1910441.
S.G. was supported in part by the NSF grant DMS-1800274.
R.Z. was supported by the NSF grant DMS-1856541, DMS-1926686 and by the Ky Fan and Yu-Fen Fan Endowment Fund at the Institute for Advanced Study.
P.Z. was supported in part by DFG (German Research Foundation) -- EXC-2047/1 -- 390685813.
The authors would like to thank Xianghong Chen, Ziyang Gao, Phil Gressman, Jonathan Pila, Chieu-Minh Tran, Shaozhen Xu, Trevor Wooley and James Wright for valuable discussions at various stages of this project.




\section{Bounding oscillatory integrals by the size of stationary sets}

In this section, we prove
Theorem \ref{201004lem2_1_general}, which says that we can bound the oscillatory integral
\begin{equation}
I(Q) = \int_{[0, 1]^d} e (Q(\xi)) d \xi
\end{equation}
by a constant times the measure of the largest ``stationary set''
\[
\epsilon := \sup_{\mu\in\R} \abs{\{\xi\in [0, 1]^d: \mu\le Q(\xi)\le \mu+1\}}
\]
if the real phase function $Q$ is semialgebraic with bounded complexity.

In order to prove the main theorem, we will change the integral $I(Q)$ to a one-dimensional integral of the form
\begin{equation}
\int_{\R} S(\beta)\cdot e(\beta) d\beta.
\end{equation}
Then we will see that the conclusion follows once we prove that $S$ changes monotonicity finitely many times.

To prove the latter property, we have to use some tools from real algebraic, and more generally,
o-minimal geometry.
Recall that a  \emph{semi-algebraic subset $S \subset \R^n$} is, by definition,
a finite union of subsets, each of which is defined by a formula of the
form $P = 0, Q_1, \ldots,Q_\ell > 0$, where each $P,Q_1,\ldots,Q_\ell \in \mathbb{R}[X_1,\ldots,X_n]$.
If the total number of polynomials occurring in the above formulas is bounded by $s$, and the maximum degree
of these polynomials is bounded by $D$, we will say that the \emph{complexity} of $S$ (as well as that of the formula) is bounded by $k = s D$.
A \emph{semi-algebraic function $f$} is a function whose graph is a semi-algebraic set.
We will say that the \emph{complexity} of $f$ is bounded by $k$ if the complexity of the semi-algebraic set $\mathrm{graph}(f)$ is bounded by $k$.

We will need to define generalizations of the notion of semi-algebraic sets and functions.
For this purpose, it is convenient to introduce some basic terminology from model theory, where these concepts originated.

\subsection{A detour into model theory}
\label{subsec:logic}
In this subsection, we introduce some terminologies and tools in model theory. They will be useful when we prove Lemma \ref{uncertaintysignchangelemma}, which is a crucial ingredient in the proof of Theorem \ref{201004lem2_1_general}.

A \emph{language}  $L$ is a triple of tuples
\[((c_i)_{i \in I}, (f_j^{(m_j)})_{j \in J}, (R_k^{(n_k)})_{k\in K}),
\]
where each $c_i, i \in I,$ is a constant symbol,
each $f_j^{(m_j)}, j\in J,$ is a function symbol of arity $m_j$, and
each $R_k^{(n_k)}, k \in K,$ is a relation symbol of arity $n_k$.

An \emph{$L$-structure} $\mathcal{M}$ is a tuple
\[
\left(
  M, (c_i^{\mathcal{M}})_{i \in I}, (f_j^{(m_j),\mathcal{M}})_{j \in J}, (R_k^{(n_k),\mathcal{M}})_{k\in K}
\right),
\]
where $M$ is a set, for each $i \in I$,  $c_i^{\mathcal{M}} \in M$, for each $j \in J$,
$f_j^{(m_j),\mathcal{M}}$ is a function $M^{m_j} \rightarrow M$, and for each $k \in K$,
$R_k^{(n_k),\mathcal{M}} \subset M^{n_k}$ is a $n_k$-ary relation on $M$.
We call
\[
(c_i^{\mathcal{M}})_{i \in I}, (f_j^{(m_j),\mathcal{M}})_{j \in J}, (R_k^{(n_k),\mathcal{M}})_{k\in K}
\]
\emph{interpretations} in $\mathcal{M}$ of the corresponding symbols in $L$.
It is a common abuse of notation to denote the structure $\mathcal{M}$ by the underlying set $M$, and
we will take this liberty.

\hide{
  An \emph{$L$-formula}, $\phi(X_1,\ldots,X_n)$ with \emph{free variables}
  $X_1,\ldots,X_n$,
  is obtained using the logic connectives $\vee,\wedge,\neg$
  (denoting disjunction, conjunction and negation, respectively),
  and the quantifiers $\exists,\forall$,
  from \emph{atomic formulas}.
  An atomic formula is a formula either of the form $t_1 = t_2$, where $t_1,t_2$ are \emph{terms}
  built out of the function and constant symbols and variables, or of the form
  $R^{(n)}(t_1,\ldots,t_n)$ where $R^{(n)}$ is an $n$-ary relation symbol and $t_1,\ldots,t_n$ are terms.
}

An \emph{$L$-formula},
is obtained using the logic connectives $\vee,\wedge,\neg$
(denoting disjunction, conjunction and negation, respectively),
and the quantifiers $\exists,\forall$,
from \emph{atomic formulas}.
An atomic formula is a formula either of the form $t_1 = t_2$, where $t_1,t_2$ are \emph{terms}
built out of the function and constant symbols and variables, or of the form
$R^{(n)}(t_1,\ldots,t_n)$ where $R^{(n)}$ is an $n$-ary relation symbol and $t_1,\ldots,t_n$ are terms.
A variable $X$ occurring in an $L$-formula $\phi$ is called a \emph{bound variable} of the formula, if it appears as $\exists X$ or $\forall X$ in $\phi$.
Otherwise, $X$ is called a \emph{free variable} of $\phi$.
If $X_1,\ldots,X_n$ are the free variables of a formula $\phi$, we will denote this by writing $\phi$ as $\phi(X_1,\ldots,X_n)$.

Given an $L$-formula $\phi(X_1,\ldots,X_n)$,
an $L$-structure $\mathcal{M} = (M,\ldots)$,
and $\bar{a} \in M^n$,
one can define in the obvious way
(using induction on the structure of the formula $\phi$)
whether $\phi(\bar{a})$ is true in the structure $M$ (denoted
$M \models \phi(\bar{a})$). Thus, every $L$-formula $\phi(X_1,\ldots,X_n)$ defines a subset, $\phi(M)$,
of $M^n$,
defined by
\[
\phi(M) = \{ (a_1,\ldots,a_n) \in M^n : M \models \phi(a_1,\ldots,a_n)\}.
\]

\begin{exam}
If $L=((0,1),(+,\cdot),(\leq))$ denotes the language of ordered fields, where $0,1$ are constant
symbols, $+,\cdot$ are binary function symbols, and $\leq$ is a binary relation symbol, then
$X_1,X_2,X_3$ are variables, $0, X_1^2 + (1+1)\cdot X_2\cdot X_1 + X_3$ are terms, and
\begin{equation}
\label{eqn:formula}
\phi(X_2,X_3) :=   (\forall X_1) (0 \leq X_1^2 + (1+1)\cdot X_2\cdot X_1 + X_3)
\end{equation}
is a formula with free variables $X_2,X_3$, and a bound variable $X_1$.

Denoting by $\mathbb{R}$ the $L$-structure whose underlying set is the real numbers, and with $0,1,+,\cdot,\leq$
interpreted in the usual way, we have
\[
\phi(\mathbb{R}) =  \{(x_2, x_3) \in \mathbb{R}^2: (\forall x_1 \in \R) 0 \leq x_1^2 + 2x_2 x_1 + x_3\}.
\]
\end{exam}

Given an $L$-structure $M$, a formula $\phi(Y_1,\ldots,Y_m,X_1,\ldots,X_n)$ with free variables
\[
Y_1,\ldots,Y_m, X_1,\ldots,X_n,
\]
and
a tuple $\bar{a} = (a_1,\ldots,a_m) \in M^m$,
we call  $\phi(\bar{a},X_1,\ldots,X_n)$ an $L(M)$-formula ($L$-formula with parameters in $M$)
with free variables $X_1,\ldots,X_n$.

\hide{
  Given an $L$-formula $\phi(X_1,\ldots,X_n)$ and $\bar{a} \in M^n$,
  one can define in the obvious way
  (using induction on the structure of the formula $\phi$)
  whether $\phi(\bar{a})$ is true in the structure $M$ (denoted
  $M \models \phi(\bar{a})$).
}

An \emph{$L$-theory $T$} is a set of of $L$-formulas without free variables  (also called \emph{$L$-sentences}).
An $L$-structure $M$ is a \emph{model of $T$} (denoted $M \models T$),
if $M\models \phi $ for each
$\phi \in T$.
An $L$-theory $T$ admits \emph{quantifier elimination} if for every $L$-formula $\phi(X_1,\ldots,X_n)$
there exists a quantifier-free $L$-formula $\psi$ such that for every model $M$ of $T$,
$M \models (\phi(X_1,\ldots,X_n) \leftrightarrow \psi(X_1,\ldots,X_n){\color{blue})}$
(here and elsewhere, $\phi \leftrightarrow \psi$ is the standard abbreviation for the formula $(\neg \phi \vee \psi) \wedge (\neg \psi \vee \phi)$).
\hide{
  A theory $T$ is \emph{consistent} if it has a model and is \emph{complete}
  if every $L$-sentence $\phi$ has
  the property that either $M \models \phi$ for every model $M$ of $T$, or $M \models \neg \phi$ for every model of
  $M$ of $T$.
}

The \emph{theory of the $L$-structure $M$}, denoted $\mathrm{Th}(M)$, is the set of $L$-sentences which are true in $M$.
Two $L(M)$-formulas $\phi(\bar{a},X_1,\ldots,X_n)$ and $\psi(\bar{b},X_1,\ldots,X_n)$ are said to be \emph{equivalent modulo $\mathrm{Th}(M)$},
if  $M \models (\forall X_1) \cdots (\forall X_n) (\phi(\bar{a},X_1,\ldots,X_n) \leftrightarrow \psi(\bar{b},X_1,\ldots,X_n))$.
The \emph{definable sets} of an $L$-structure $M$ are sets of the form
\[
\{(x_1,\ldots,x_n) \in M^n : M \models \phi(x_1,\ldots,x_n)\},
\]
where $\phi$ is an $L(M)$-formula.
The \emph{definable functions} are those whose graphs are definable sets.

For the rest of the article, we will denote by
$L=((0,1),(+,\cdot),(\leq))$ the \emph{language of ordered fields}. 

The set $\mathbb{R}$ with the usual interpretations of $0,1,+,\cdot,\leq$ is an $L$-structure.
\hide{
  The formulas  $(X_1 + X_2)\cdot X_3 = 1$, $X_1 \leq X_2$ are atomic formulas in the
  language $L$.

  If $T = \mathrm{Th}(\mathbb{R})$, every atomic formula
  is equivalent  modulo $T$ to a
  a formula of the type $P \leq 0, P \in \mathbb{Z}[X_1,\ldots,X_n]$.
  It is an easy consequence that each $L(\mathbb{R})$-formula
  $\phi(X_1,\ldots,X_n)$ with free variables $X_1,\ldots,X_n$,
  is equivalent modulo $T$ to a formula
  obtained using logical connectives $\vee,\wedge,\neg$ and the quantifiers $\exists,\forall$,
  from formulas of the type $P \leq 0, P \in \mathbb{R}[X_1,\ldots,X_n]$.
  Thus, $\phi(X_1,X_2) = (\exists X_3) (X_3^2 + 2X_1 X_3 + X_2 \geq 0)$
  is an example of an $L(\mathbb{R})$-formula.
}
Notice that, by definition, semi-algebraic sets are definable sets in the $L$-structure $\mathbb{R}$.
The Tarski-Seidenberg theorem \cite{Tarski} states that the
the $L$-theory, $\mathrm{RCOF}$,  of \emph{real closed ordered fields} admits quantifier elimination.
Here, $\mathrm{RCOF}$ is the set of $L$-sentences axiomatizing ordered fields, in which every
non-negative element is a square and every polynomial having odd degree has a root.
By definition, we see that the field $\R$ (with the usual interpretations of $0,1,+,\cdots,\leq$) is a model of $\mathrm{RCOF}$.
As a consequence of the Tarski-Seidenberg theorem \cite{Tarski}, one
immediately obtains that every definable set of the $L$-structure $\mathbb{R}$ is semi-algebraic.

As an illustration
of Tarski-Seidenberg theorem, observe that the $L$-formula $\phi$ in \eqref{eqn:formula} (which has quantifiers) is equivalent modulo
the theory $\mathrm{RCOF}$ to the quantifier-free $L$-formula
\[
\psi(X_2,X_3) : = X_2 \cdot X_2 \leq X_3.
\]
As a consequence,
\[
\phi(\R) = \psi(\R) = \{(x_2,x_3) \in \R^2 : x^2 \leq x_3 \}.
\]
Another simple geometric consequence of the Tarski-Seidenberg theorem is
that if $S_1$ and $S_2$ are a semialgebraic sets in $\R$, then the set $\{x \in \R  : \exists y \in S_1 \text{ s.t. } x+y \in S_2\}$ is also semialgebraic.


While the Tarski-Seidenberg theorem is not quantitative as stated above, there are effective versions with
complexity estimates. In particular, given any $L$-formula $\phi$, there exists a quantifier-free $L$-formula
$\psi$ equivalent to $\phi$ modulo $\mathrm{RCOF}$, such that the
number and degrees of the polynomials appearing in $\psi$ is bounded by a function of the
number and degrees of the polynomials appearing in $\phi$, as well as the number of quantified and free variables.
This is in fact a consequence of the original proof of Tarski \cite{Tarski} of the Tarski-Seidenberg theorem, and we omit detailed explanation here.
We remark that the above mentioned function can be  made explicit, see for instance \cite[Theorem 14.16]{MR2248869}.

\subsubsection{o-minimal structures}
\label{subsubsec:o-minimal}
Let $L'$ be any language containing the binary relation $\leq$, and let $M$ be an $L'$-structure such that the relation $\leq$ defines a total order on $M$.
The structure $M$ is called \emph{o-minimal} if the set of definable subsets of $M$ is the smallest possible --
namely, the set whose elements are finite unions
of points and intervals (such subsets are clearly definable by quantifier-free $L'(M)$-formulas).

\begin{exam}
Recall that $L$ is the language of ordered fields.
The $L$-structure $\mathbb{R}$ is an example of an o-minimal structure.
\end{exam}

Suppose that $L_{+}$ is a language containing $L$ (we call $L_{+}$ an
\emph{expansion} of $L$). An $L_{+}$-structure is called an \emph{o-minimal expansion of $\mathbb{R}$}
if the underlying set of the structure is $\mathbb{R}$, and the structure is o-minimal (i.e., the definable subsets
of $\mathbb{R}$ are precisely the semi-algebraic subsets of $\mathbb{R}$).

Let us look at an example to better digest this notion. In the following example, we expand the language $L$ with one additional function symbol.

\begin{exam}
Let $L_{+}$ be the expansion of $L$ by an additional $1$-ary function
symbol $f$, and consider the $L_{+}$-structure $\mathbb{R}_{\sin}$ whose underlying set is $\mathbb{R}$,
with the symbols of $L$ interpreted as before,
and the $1$-ary function symbol $f$  interpreted by the trigonometric function $\sin: \mathbb{R} \rightarrow \mathbb{R}$.
Then, $\mathbb{R}_{\sin}$ is \emph{not} an o-minimal expansion of $\mathbb{R}$,
since the definable subset of $\mathbb{R}$ defined by
the formula $f(X)=0$ is the set $\{n\pi: n \in \mathbb{Z}\}$,
which cannot be expressed as a \emph{finite} union of points and intervals of $\mathbb{R}$.
\end{exam}

\subsubsection{The o-minimal structures $\mathbb{R}_{\mathrm{an}}$ and $\mathbb{R}_{\mathrm{an},\exp}$}
\label{subsubsec:an-exp}
The above example shows that the real analytic functions cannot all be definable functions in any o-minimal expansion of
$\mathbb{R}$.
For instance, the function $\sin(x)$ is never
definable in any o-minimal expansion of $\mathbb{R}$.

However, notice that if we consider a function $f$,  which is equal to $\sin(x)$ on $[-1,1] \subset \mathbb{R}$, and defined to be $0$ outside of the
$[-1,1]$, then the set of zeros of $f$ is a semi-algebraic set
and hence definable in every o-minimal expansion of $\mathbb{R}$.

Motivated by this example, we call $f:\mathbb{R}^n \rightarrow \mathbb{R}$
a \emph{restricted analytic function}, if there is an open subset $U$ containing $[-1,1]^n$ in $\mathbb{R}^n$, and an analytic function $g:U \rightarrow \mathbb{R}$, such that
\begin{equation}
f(x) =
\begin{cases}
g(x) & \mbox{ for } x \in [-1,1]^n,\\
0 & \mbox{ otherwise}.
\end{cases}
\end{equation}
Let $L_{\mathrm{an}}$ denote the expansion of $L$ obtained by including a function symbol (of appropriate arity)
for each restricted analytic function. We denote by $\mathbb{R}_{\mathrm{an}}$ the $L_{\mathrm{an}}$
structure with underlying set $\mathbb{R}$, with the usual interpretations
of the symbols of $L$, and where we interpret the new symbols by the corresponding restricted analytic functions.
It is a theorem of Gabrielov \cite{Gabrielov} that
the $L_{\mathrm{an}}$-structure $\mathbb{R}_{\mathrm{an}}$ is an o-minimal expansion of $\mathbb{R}$.

The (unrestricted) exponential function is not a restricted analytic function. We expand the language
$L_{\mathrm{an}}$ further and include a symbol $\exp$ for the exponential function and denote the new language
by $L_{\mathrm{an},\exp}$. Denote by $\mathbb{R}_{\mathrm{an},\exp}$ the $L_{\mathrm{an},\exp}$-structure
obtained by interpreting the new symbol $\exp$ by the exponential function, and interpreting the symbols of
$L_{\mathrm{an}}$ as before. It is a theorem of van den Dries, Macintyre, and Marker  that:

\begin{thm}[\cite{MR1289495}]
The $L_{\mathrm{an},\exp}$-structure $\mathbb{R}_{\mathrm{an},\exp}$ is an o-minimal expansion of
$\mathbb{R}$.
\end{thm}

\begin{remark}
\label{rem:logarithm}
Notice that the logarithm function $\log: (0,\infty)\rightarrow \mathbb{R}$ is definable in $\mathbb{R}_{\mathrm{an},\mathrm{exp}}$, since its graph
\[
\mathrm{graph}(\log) = \{(x,y) \in \mathbb{R}^2 \mid x > 0, x = \mathrm{exp}(y) \}
\]
is definable in $\mathbb{R}_{\mathrm{an},\mathrm{exp}}$.
\end{remark}

Let $\mathcal{S}$ denote an o-minimal expansion of $\mathbb{R}$.
If we denote by $\mathcal{S}_n$ for each $n \geq 0$ the
set of subsets of $\mathbb{R}^n$ which are definable in $\mathcal{S}$,
then the tuple $(\mathcal{S}_n)_{n \geq 0}$ satisfy the following properties,
which follow directly from the definition of an o-minimal expansion.

\begin{enumerate}[(a)]
\item
\label{itemlabel:1}
If $n \geq 0$ and $A,B \in \mathcal{S}_n$, then $A\cup B,A \cap B, \mathbb{R}^n - A \in \mathcal{S}_n$.
If $n,m\geq 0$ and $A \in \mathcal{S}_m, B \in \mathcal{S}_n$, then $A \times B \in \mathcal{S}_{m+n}$.
\item
\label{itemlabel:2}
If $\pi:\mathbb{R}^m \times \mathbb{R}^n \to \mathbb{R}^m$ is the projection map and $A \in \mathcal{S}_{m+n}$, then $\pi(A) \in \mathcal{S}_m$.
\item
\label{itemlabel:3}
Elements of $\mathcal{S}_1$ are precisely the finite unions of points and open intervals.
\end{enumerate}

\begin{remark}
\label{rem:constructible}
While the set of definable sets in an o-minimal expansion of $\mathbb{R}$ is stable under finite
set operations and projections (Properties \eqref{itemlabel:1} and \eqref{itemlabel:2} above),
the set of definable functions is not necessarily closed under taking parametric integrals.
More precisely, if $f:\R^m \times \R^n \rightarrow \R$ is a definable function in $\mathcal{S}$,
then the function
\begin{equation}\label{210205e2_4}
g(x) =
\begin{cases}
\int_{\R^m} f(y,x) dy &  \mbox{ if } \int_{\R^m} |f(y,x)| dy < \infty, \\
0, & \mbox{ otherwise }
\end{cases}
\end{equation}
is not necessarily definable in $\mathcal{S}$. It suffices to take the example of the function
\begin{equation}
f(y, x) =
\begin{cases}
1/y &  \text{ if } x>0, 1<y<x, \\
0 & \text{ otherwise,}
\end{cases}
\end{equation}
whose graph is a semi-algebraic set and so is clearly definable in the o-minimal structure $\mathbb{R}$. However,
the parametric integral is the logarithm function whose graph is not a semi-algebraic and hence
is not definable in $\mathbb{R}$.
On the other hand, notice that $g(x)$ is definable in the o-minimal expansion
$\mathbb{R}_{\an,\exp}$ (cf. Remark~\ref{rem:logarithm}).

In fact, it is proved in \cite[Theorem 1.3]{MR2769219}, that if $f:\R^m \times \R^n \to \R$ is a \emph{constructible} function, then the function defined in \eqref{210205e2_4}
is constructible (cf. Definition 1.2 in \cite{MR2769219}).
We omit the definition of constructible functions in the sense of \cite{MR2769219}, but
list two of their properties, which are the only facts that we will need in this paper.
\begin{enumerate}
\item
Constructible functions are definable in the o-minimal expansion $\mathbb{R}_{\an,\exp}$ of $\mathbb{R}$. This follows easily from the definition of constructible functions and the fact that the logarithm function is definable in $\mathbb{R}_{\an,\exp}$.
\item
Semi-algebraic functions are constructible.
\end{enumerate}
\end{remark}

An important consequence of the properties \eqref{itemlabel:1}, \eqref{itemlabel:2} and \eqref{itemlabel:3}
is the topological tameness of definable sets in an o-minimal expansion of $\mathbb{R}$.
Let $\mathcal{S}$ be a $L_{\mathcal{S}}$-structure which is an o-minimal expansion of $\mathbb{R}$.
Then the definable sets of $\mathcal{S}$ can only have a finite number of connected components
(this is a consequence of  \cite[Theorem 2.9]{MR1633348}).

Another important property of the definable sets is that of topological tameness in families.
\begin{prop}[{\cite[Chapter 9, Section 2]{MR1633348}}]\label{prop:tame}
Let $\mathcal{S}$ be a $L_{\mathcal{S}}$-structure which is an o-minimal expansion of $\mathbb{R}$.
Let $X$ and $Y$ be definable sets and let $f:X\to Y$ be a definable function.
Then, there are finitely many homeomorphism types amongst the fibers $X_y = X \cap f^{-1}(y), y \in Y$.
In particular, there exists a constant $N_f \geq 0$,
such that the number of connected components of $X_y, y \in Y$ is bounded by $N_f$.
\end{prop}

As an application of Proposition~\ref{prop:tame}, we have:
\begin{prop}
\label{prop:monotonicity}
Let $\mathcal{S}$ be a $L_{\mathcal{S}}$-structure which is an o-minimal expansion of $\mathbb{R}$.
Let $h: \mathbb{R}^D\times \mathbb{R} \rightarrow \mathbb{R}$ be a definable function, and for
$x \in \mathbb{R}^D$, let $N(x)$ denote the number of times the function $h(x,\cdot):\mathbb{R} \rightarrow \mathbb{R}$ changes monotonicity. Then, $\sup_{x \in \mathbb{R}^D}  N(x) < \infty$.
\end{prop}
\begin{proof}
First observe that for each $x \in \R^D$, the function $h(x,\cdot)$ is a definable function. Using the monotonicity theorem for definable functions \cite[Theorem 1.2, Chapter 3]{MR1633348}, for each $x \in \R^D$, there exists a finite partition of $\R$ into points and open intervals, such that over each open interval the function $h(x,\cdot)$ is either constant, strictly increasing, or strictly decreasing.
Hence, the set of points at which the function $h(x,\cdot)$ changes monotonicity is contained in the set of end points of these intervals, and in particular is a finite subset of $\R$.

Now, let $S = \{(x,t) \in \mathbb{R}^D \times \mathbb{R}\mid h(x,\cdot) \mbox{ changes monotonicity at } t\}$.
Since changing monotonicity of a function is expressible using a first-order formula (involving quantifiers) in the language $L_{\mathcal{S}}$,
it follows that the set $S$ is definable in o-minimal structure $\mathcal{S}$.

Now, applying Proposition~\ref{prop:tame}, we get that there are finitely many homeomorphism types
amongst the fibers of the map $(x,t) \mapsto x$ restricted to $S$, noting that this map is clearly definable in
$\mathcal{S}$. Moreover, each fiber has finitely many points using the observation in the first paragraph of the proof. The proposition follows.
\end{proof}

\subsection{Proof of Theorem \ref{201004lem2_1_general}}

With the preparation above, let us prove a lemma that will play a crucial role in the proof of Theorem \ref{201004lem2_1_general}.
\begin{lem}\label{uncertaintysignchangelemma}
Let $d \geq 1$ and let $U$ be a semialgebraic set in $\R^{d+1}$ with complexity $\leq k$.
Assume further that $U$ is bounded.
For each $\beta \in \R$, we define
\begin{equation}
A_U(\beta) = \int_{\R^d} \one_U (\xi, \beta) d\xi.
\end{equation}
Then the function $A_U$ changes monotonicity $O_{d, k} (1)$ times.
\end{lem}

\begin{proof}[Proof of Lemma \ref{uncertaintysignchangelemma}]
First observe that for each $k,d \geq 0$
there exists $D= D(k,d), N = N(k,d)$, and
a finite set
\[
\Phi_{k,d} =\{\phi_1(Y_1,\ldots,Y_D,X_1,\ldots,X_{d+1}), \ldots, \phi_N(Y_1,\ldots,Y_D,X_1,\ldots,X_{d+1})\}
\]
where each $\phi_i$ is a quantifier-free $L$-formula, such that if $U \subset \R^{d+1} $
is a semi-algebraic set  of complexity bounded by $k$, then
there exists $1 \leq i \leq N$ and  $\eta \in \R^{D}$
such that
\[
U = \{(x_1,\ldots,x_{d+1}) \in \R^{d+1} \mid \R \models \phi_i(\eta,X_1,\ldots,X_{d+1})\}.
\]

Intuitively, the above property holds because by definition, our set $U$ 
can be described by an $L(\R)$-formula involving $O_{k, d}(1)$ many polynomials of degree $O_{k, d}(1)$, parameters in $\R$, logic connectives and constant symbols.

For each $i, 1 \leq i \leq N$, let
$V_i \subset \R^{D} \times \R^{d+1}$ be the set defined by $\phi_i$, and let
$\pi: \R^{D} \times \R^{d + 1} \rightarrow \R^{D}$ be the projection map to the first factor.

For $1 \leq i \leq N$,
consider the functions $A_i$ on $\R^{D+1} = \R^D \times \R$ defined as follows: For $\eta\in \R^D$ and $\beta\in \R$, we set
\begin{equation}
A_i(\eta, \beta):=\int_{\R^{d}} \one_{V_{i}}(\eta,\xi,\beta) d\xi
\end{equation}
whenever the above integral is finite, and $A(\eta, \beta)=0$ otherwise.

The functions $\one_{V_{i}}$ are semialgebraic and thus definable
in
$\R_{\mathrm{an}}$.
By Theorem 1.3 in \cite{MR2769219} (as mentioned in Remark \ref{rem:constructible}), $A_i(\eta, \beta)$ is a constructible function on $\R^{D+1}$.
Using the properties of constructible functions listed in Remark~\ref{rem:constructible},
we obtain that $A_i(\eta, \beta)$ is definable in $\R_{\an, \exp}$.

It now follows from Proposition~\ref{prop:monotonicity} that there exists $M_i = M_i(k,d) > 0$, such that
the number of times the function $A_i(\eta,\cdot), \eta \in \R^D$ changes monotonicity is bounded by $M_i$.
Let $M(k,d)  = \max_{1 \leq i \leq N} M_i$.
It follows that the number of times the function $A_U$ changes monotonicity is bounded by
$M(k,d)$, observing that there exists $i, 1\leq i \leq n$,
$\eta \in \R^D$, such that $A_U(\beta) = A_i(\eta,\beta)$ for $\beta \in \R$.
\end{proof}

With Lemma \ref{uncertaintysignchangelemma} in mind, we now prove Theorem \ref{201004lem2_1_general}.

\begin{proof}[Proof of Theorem \ref{201004lem2_1_general}]
Fix a constant $0<\delta_0< \frac{1}{2}$ (in our proof $\delta_0 = 0.1$ will work). In the proof we will suppress the dependence on the constant $\delta_0$.
By a change of variables, we have
\begin{multline*}
\Big(\int_{-\delta_0}^{\delta_0} e(\alpha)d\alpha\Big)\Big(\int_{[0, 1]^d} e(Q(\xi))d\xi\Big)
=
\int_{[0, 1]^d} \int_{-\delta_0}^{\delta_0} e(Q(\xi)+ \alpha) d\alpha d\xi
\\ =
\int_{[0, 1]^d} \int_{Q(\xi)-\delta_0}^{Q(\xi)+\delta_0} e(\beta) d\beta d\xi
=
\int_{\R} |\{\xi \in [0, 1]^d: Q(\xi) \in [\beta - \delta_0, \beta + \delta_0]\}|\cdot e(\beta) d\beta.
\end{multline*}
Here, $|\cdot|$ denotes the Lebesgue measure.
Note that $\int_{-\delta_0}^{\delta_0} e(\alpha)d\alpha$ is a nonzero constant (depending only on $\delta_0$) as $0 < \delta_0 < \frac{1}{2}$.
Hence, if we define a \emph{stationary set} to be
\begin{equation}
S_{Q; \delta_0} (\beta) = \{\xi \in [0, 1]^d: Q(\xi) \in [\beta - \delta_0, \beta + \delta_0]\},
\end{equation}
then we have
\begin{equation}
\Big|\int_{[0, 1]^d} e(Q(\xi))d\xi\Big| \sim \Big|\int_{\R} |S_{Q; \delta_0} (\beta)| e(\beta) d\beta\Big|.
\end{equation}
Since $Q$ is bounded, the function $|S_{Q; \delta_0}(\cdot)|$ is compactly supported. We claim that it is piecewise monotonic and that it only changes monotonicity $O_{d, k} (1)$ many times.
This claim follows from Lemma \ref{uncertaintysignchangelemma} in a straightforward way, and we postpone its justification to the end of this proof.

We now explain how to get the desired conclusion from the above claim.
We decompose $\R$ into a union of $O_{d, k} (1)$ many closed intervals $\bigcup_j I_j$ such that the interiors of $I_j$ are disjoint and that the function $|S_{Q; \delta_0}(\cdot)|$ is compactly supported and monotonic on each $I_j$.
For each $I_j$, the total variation of $|S_{Q; \delta_0}(\cdot)|$ on $I_j$ is bounded by $2\sup_{\beta\in \R} |S_{Q; \delta_0} (\beta)|$, uniformly in $j$. By integration by parts, we see
\begin{equation}
\Big|\int_{I_j} |S_{Q; \delta_0} (\beta)| e(\beta) d\beta\Big| \lesssim  \sup_{\beta\in \R} |S_{Q; \delta_0} (\beta)|,
\end{equation}
uniformly in $j$.
Hence, we get
\begin{equation}
\Big|\int_{\R} |S_{Q; \delta_0} (\beta)| e(\beta) d\beta\Big| \leq \sum_j \Big|\int_{I_j} |S_{Q; \delta_0} (\beta)| e(\beta) d\beta\Big| \lesssim_{d, k} \sup_{\beta\in \R} |S_{Q; \delta_0} (\beta)|\lesssim_{d, k} \epsilon
\end{equation}
by assumption. This is the desired conclusion.

Finally, we justify our claim that $|S_{Q; \delta_0}|$  is piecewise monotonic and that it only changes monotonicity $O_{d, k} (1)$ many times. Consider the set
\begin{equation}
\Gamma (Q) :=
\{(\xi, \beta) \in \R^{d}\times\R: \beta - \delta_0 \leq Q(\xi)\leq \beta + \delta_0, \xi \in [0, 1]^d\}.
\end{equation}
The set $\Gamma(Q)$ is bounded and can be described as a ``vertical'' neighborhood of the graph of $Q$.

\hide{
  {\color{brown}By quantifier elimination (i.e. Tarski's theorem), we see that $\Gamma(Q)$ is semialgebraic with complexity $O_{d, k} (1)$.} Now we apply Lemma \ref{uncertaintysignchangelemma} to the set $\Gamma(Q)$ and note that $S_{Q; \delta_0} (\beta) = A_{\Gamma (Q)} (\beta)$ in the notation of that lemma. By the conclusion of Lemma \ref{uncertaintysignchangelemma} we  deduce the claim in the beginning of this paragraph.
}
Observe that since $Q$ is a semi-algebraic function, $\mathrm{graph}(Q)$ is a semi-algebraic subset of
$\R^{d+1}$ and so defined by a $L(\mathbb{R})$-formula $\phi(X_1,\ldots,X_{d+1})$. Then
$\Gamma(Q)$ is defined by the formula
\[
\psi(X_1,\ldots,X_{d+1}) := (\exists Y) \phi(X_1,\ldots,X_{d},Y) \wedge (X_{d+1} \leq Y +\delta_0) \wedge (Y \leq X_{k+1} + \delta_0).
\]
Using an effective version of Tarski-Seidenberg theorem 
(that was mentioned in Subsection 2.1; see e.g. \cite[Theorem 14.16]{MR2248869} for a reference),
we obtain that $\psi$ is equivalent to a quantifier-free
formula $\widetilde{\psi}$ such that the number of polynomials and their degrees that occur in $\widetilde{\psi}$
is $O_{d,k}(1)$. And so $\Gamma(Q)$ is semialgebraic with complexity $O_{d, k} (1)$. Now we apply Lemma \ref{uncertaintysignchangelemma} to the set $\Gamma(Q)$, and note that $S_{Q; \delta_0} (\beta) = A_{\Gamma (Q)} (\beta)$ in the notation of that lemma. By the conclusion of Lemma \ref{uncertaintysignchangelemma}, we deduce the claim in the beginning of this paragraph.
\end{proof}

\begin{remark}
In Theorem \ref{201004lem2_1_general}, the amplitude  function of the oscillatory integral is taken to be the sharp cut-off $\one_{[0, 1]^d}$. For a compactly supported smooth amplitude function, one can apply a Fourier series expansion and then apply  Theorem \ref{201004lem2_1_general}. In applications, due to the rapid decay of Fourier coefficients, the extra summation that is brought in by Fourier series expansion usually will not cause any problem. For example, in Tarry's problem, one can insert a compactly supported smooth amplitude function in \eqref{sharp_cutoff}, and ask what the convergence exponent is. It is elementary to see that it coincides with the one of sharp cut-off, that is, $p_{2, k}$.
\end{remark}






\section{Theorem~\ref{main_tarry}: Reduction to a spread out case}
\label{sec:tarry-concentrated}

Sections~\ref{sec:tarry-concentrated} and \ref{sec:tarry-spread} are devoted to the proof of Theorem~\ref{main_tarry}.
From now on, we will fix $d=2$ and abbreviate $p_{2, k}$ to $p_k$. To simplify notation, we denote
\begin{equation}
\qk=\frac{1}{6}k(k+1)(k+2)+2.
\end{equation}
Under this notation, to prove Theorem \ref{main_tarry}, it is equivalent to prove that $p_k=\qk$ for every $k\ge 2$.\\

We will use $\extend(x)$ to denote the function $E_{[0, 1]^d}^{(d, k)} 1(x)$. Our goal is to show that
\begin{equation}
|\{x\in \R^N: |E_k 1|(x)\in [\delta, 2\delta]\}|
\lesim_{k}
\abs{\log_{+} \delta}^{k-1} \delta^{-\qk}
\end{equation}
for every $\delta\leq 1$, where
\begin{equation}
\log_+\delta:=\log \big(1+\frac{1}{\delta}\big).
\end{equation}
This will imply the upper bound of $p_{k}$ that $p_k\le \qk$.
We will use $P(\xi; x)$ to denote the polynomial $x\cdot \phi_{k}(\xi)$.
For a polynomial $P: [0, 1]^2\to \R$ and real numbers $c>0$ and $\mu$, define
\begin{equation}
Z_c(P, \mu):=\{\xi\in [0, 1]^2: \mu\le P(\xi)\le \mu+c\}.
\end{equation}
For a given $x$, since the measure of $Z_1(P(\cdot\ ; x), \mu)$ depends continuously on $\mu$, we can choose $\mu_x$ such that
\begin{equation}\label{201219e7_2}
|Z_1(P(\cdot\ ; x), \mu_{x})| = \sup_{\mu\in\R} |Z_1(P(\cdot\ ; x), \mu)|.
\end{equation}
To simplify notation, we denote \begin{equation}
\gammax:=Z_1(P(\cdot\ ; x), \mu_x).
\end{equation}
From Theorem  \ref{201004lem2_1_general}, we immediately obtain that there exists $C_k>0$ such that for every $\delta\leq 1$, it holds that
\begin{equation}\label{201118e3_5}
\{x\in \R^N: |E_k 1|(x)\in [\delta, 2\delta]\}
\subseteq L_{\delta/C_k},
\end{equation}
where
\[
L_{\delta} := \{x\in \R^N: |\gammax|> \delta\}.
\]
We will show that there exists $C_{k}<\infty$ such that
\begin{equation}
\label{eq:Ldelta-est}
|L_{\delta}|\le C_{k} \abs{\log_{+} \delta}^{k-1} \delta^{-\qk},
\end{equation}
for every $\delta \in (0,1]$.
This will be proven via an inductive argument.
More precisely, we will show that there exists $C_{k, 1}>0$ such that for every sufficiently large dyadic number $K\geq 1$ (picking $K\ge 2^{10k}$ will be enough), we can find another constant $C_{k, K}$ such that
\begin{equation}\label{eq:Ldelta-recursion}
\abs{L_{\delta}} \leq
K \cdot K^{k(k+1)(k+2)/6} \abs{L_{\delta K/C_{k, 1}}}
+
C_{k,K} \abs{\log_{+} \delta}^{k-1} \delta^{-k(k+1)(k+2)/6-2}.
\end{equation}
Since $L_{1} = \emptyset$, this implies our claim \eqref{eq:Ldelta-est} after choosing $K$ so large that
\[
K^{k(k+1)(k+2)/6+1} < (K/C_{k, 1})^{k(k+1)(k+2)/6+2}
\]
and unrolling the recursion \eqref{eq:Ldelta-recursion}.

\begin{rem}
The main contribution in the recursive estimate \eqref{eq:Ldelta-recursion} comes from the second summand on the right-hand side.
This is due to the fact that the power of $K$ in the first summand is strictly smaller than $p_{k}$.
It is this fact that allows us to use a fixed value of $K$ and leads to only a logarithmic loss in \eqref{eq:Ldelta-est}.
\end{rem}

\begin{remark}
The choice of $C_{k, 1}$ will be made in Lemma~\ref{lem:large-projection}.
\end{remark}

Let $K\gg 1$ be large dyadic number to be determined.
Let $\mc{S}_K$ be the collection of all (vertical) dyadic rectangles $S_K\subset [0, 1]^2$ of size $1\times 1/K$.
We have $|\mc{S}_K|=K$.
Let $C_{k, 1}>1$ be a large constant to be determined. We write $L_{\delta}$ as a disjoint union
\begin{equation}
L_{\delta}=L_{\delta, \mathrm{con}}\cup L_{\delta, \mathrm{sprd}},
\end{equation}
where
\begin{equation}
\begin{split}
\quad
L_{\delta}(S_{K}) := \{x\in L_{\delta}: |\gammax\cap S_K| > \delta/C_{k, 1}\}.
\end{split}
\end{equation}
\[
L_{\delta, \mathrm{con}}:=
\bigcup_{S_{K} \in \mc{S}_{K}} L_{\delta}(S_{K}),
\]
We will control the sizes of $L_{\delta, \mathrm{con}}$ and $L_{\delta, \mathrm{sprd}}$ separately. For $x \in L_{\delta}$, we say that it is in the \emph{concentrated case} if $x \in L_{\delta, \mathrm{con}}$ and that it is in the \emph{spread out case} if $x \in L_{\delta, \mathrm{sprd}}$.

\begin{rem}
The concentrated/spread out dichotomy here is inspired by the broad-narrow dichotomy of Bourgain-Guth in \cite{MR2860188}. We leave the comparison to the interested readers.
\end{rem}


In the remaining part of this section we will control $L_{\delta, \mathrm{con}}$.
Let us start with estimating $|L_{\delta}(S_K)|$ for a fixed $S_K$.
Without loss of generality, take $S_K=[0, 1/K]\times [0, 1]$.
For an $x\in \R^N$, define
\begin{equation}
\bar{x}:=(x_{\beta}/K^{\beta_{1}})_{1\le |\beta|\le k},
\end{equation}
with $\beta=(\beta_1, \beta_2)\in \N^2$, so that $P((\xi_{1},\xi_{2}); x) = P((K\xi_{1},\xi_{2}); \bar{x})$.
Then, for $x\in L_{\delta}(S_{K})$, we have
\begin{equation}
|Z_1(P(\cdot\ ; \bar{x}), \mu_x)|
=
K |\gammax \cap S_{K}|
>
K \delta/C_{k, 1}.
\end{equation}
By definition of $L_{\cdot}$, we have
\begin{equation}
\{\bar{x}\in \R^N: |Z_{\bar{x}}| > K\delta/C_{k, 1} \}
= L_{K\delta/C_{k, 1}}.
\end{equation}
By scaling, this implies
\begin{equation}
|L_{\delta}(S_K)|
\le
\abs{L_{K\delta/C_{k, 1}}} \prod_{\abs{\beta}\leq k} K^{\beta_{1}}
=
K^{k(k+1)(k+2)/6} \abs{L_{K\delta/C_{k, 1}}}.
\end{equation}
We sum over $S_K$ and obtain
\begin{equation}
|L_{\delta,\mathrm{con}}|
\le
K \cdot K^{k(k+1)(k+2)/6} \abs{L_{K\delta/C_{k, 1}}}.
\end{equation}
This finishes the argument for the concentrated case.

\section{Theorem~\ref{main_tarry}: The spread out case}
\label{sec:tarry-spread}

In this section, we will show that
\begin{equation}
\abs{L_{\delta, \mathrm{sprd}}}\le C_{k,K} \abs{\log_{+} \delta}^{k-1} \delta^{-k(k+1)(k+2)/6-2},
\end{equation}
thus finishing the proof of the desired estimate \eqref{eq:Ldelta-recursion}. Next we sketch the proof ideas before working out the details,

For $x \in L_{\delta, \mathrm{sprd}} \subset L_{\delta}$, recall that the stationary set $Z_x = \{\xi \in [0, 1]^2: \mu_x \leq P(\xi; x) \leq \mu_x + 1\}$ has measure at least $\delta$. Let us think about what such a $Z_x$ may look like. Note that $Z_x$ has bounded complexity. In the second subsection, we will prove that a semialgebraic set of bounded complexity in $[0, 1]^d$ that has measure at least $\delta > 0$ is ``not far from'' a union of $\delta$-cubes. This is essentially Lemma \ref{lem:lagrange-interpolation} in two dimensions.\footnote{There is a higher dimensional version that can be proved in an identical way. We omit that more general version here since it is irrelevant to our current problem.} Before proving such a lemma we will need some preparations (done in the first subsection) using elementary real algebraic geometry.

Since $Z_x$ has bounded complexity, the (quantitative) Tarski-Seidenberg theorem shows that its projection to the $\xi_1$-axis is a union of boundedly many points and intervals. If we pretend that $Z_x$ is a union of $\delta$-squares, by the fact that it is spread out in the $\xi_1$ direction (thus hitting many vertical $1/K$-strips), we see that the projection of $Z_x$ to the $\xi_1$-axis has to contain a ``long interval'' (with constant length)  if we choose $C_{k, 1}$ large enough. This means we can find an interval $I \subset [0, 1]$ on the  $\xi_1$-axis of constant length such that for all dyadic interval $\tilde{I}$ of length $\delta$ in $I$, there is a $\delta$-square $\Delta_{\tilde{I}}$ in $Z_x$ whose projection to the $\xi_1$-axis  is $\tilde{I}$. In real life,  $Z_x$ is not honestly a union of $\delta$-squares but we still have something very similar to the conclusion above thanks to Lemma \ref{lem:large-projection}.

Finally, we fix a constant ($k+1$, to be more accurate) many evenly-spaced $\tilde{I}$ and look at all possible tuples $(\Delta_{\tilde{I}})_{\tilde{I}}$. For each tuple, if every $\delta$-square in it is essentially contained in $Z_x$, we will see that $x$ has to be contained in a very specific rectangular box (depending on the particular tuple $(\Delta_{\tilde{I}})_{\tilde{I}}$). Geometrically, this is the dual box to the convex hull of all caps corresponding to all $\Delta_{\tilde{I}}$ in that tuple. Therefore the union of all such rectangular boxes will contain $L_{\delta, \mathrm{sprd}}$. Simply by adding up the volume of all the rectangular boxes we get the desired upper bound of $\abs{L_{\delta, \mathrm{sprd}}}$ and finish the proof. This is carried out in the last subsection.

\begin{rem}
For general $S_{d, k}$ when $d \geq 3$, we expect this proof framework to continue working but also expect a lot more new technical difficulties and hence do not pursue the more general problem here. See Remark \ref{rem:higher-dim-tarry} for a discussion.
\end{rem}

\subsection{Shifts of semialgebraic sets}

\begin{lemma}[Finite boundary]
\label{lem:finite-boundary}
For every $\kappa\in\N$ there exists $B_{\kappa}<\infty$ such that every semialgebraic subset $\Gamma \subseteq \R$ of complexity $\leq \kappa$ has at most $B_{\kappa}$ boundary points.
\end{lemma}

\begin{proof}[Proof of Lemma \ref{lem:finite-boundary}]
By definition, $\Gamma$ is a union of $O_{\kappa} (1)$ many subsets $\Gamma_j$ where each $\Gamma_j$ is a set defined by $O_{\kappa} (1)$ many polynomial equations or inequalities of degree $O_{\kappa} (1)$. This implies that each $\Gamma_j$ is an intersection of $O_{\kappa} (1)$ many larger sets, each larger set being a union of $O_{\kappa} (1)$ many points or intervals. Hence each $\Gamma_j$ is a union of $O_{\kappa} (1)$ many points or intervals. Taking the union, we deduce that $\Gamma$ is also a union of $O_{\kappa} (1)$ many points or intervals and the conclusion follows.
\end{proof}

\begin{corollary}[Shifts of a bounded complexity set]
\label{cor:shift-of-bdd-complexity}
Let $\Gamma \subseteq [0,1]$ be a semialgebraic set of complexity $\leq \kappa$.
Then, for every $\delta > 0$, we have
\[
\abs{ (\Gamma + [-\delta,\delta]) \setminus \Gamma } \lesssim_{\kappa} \delta.
\]
\end{corollary}
\begin{proof}[Proof of Corollary \ref{cor:shift-of-bdd-complexity}]
By Lemma~\ref{lem:finite-boundary}, $\Gamma$ is the union of $O_{\kappa} (1)$ many intervals, which immediately implies the desired bound.
\end{proof}

\begin{lemma}
\label{lem:large-intersection-of-shifts}
For every $k,\kappa\in\N$, there exists a small constant $c_{k,\kappa}>0$ such that, for every $\delta>0$ and every semi\-algebraic set $\Gamma \subseteq [0,1]^{2}$ with complexity $\leq\kappa$ and measure $\abs{\Gamma}>\delta$, we have
\[
\abs[\Big]{\bigcap_{a \in \Set{0,\dotsc,k}^{2}} (\Gamma - c_{k,\kappa}\delta a)} > \delta/2.
\]
\end{lemma}
\begin{proof}[Proof of Lemma \ref{lem:large-intersection-of-shifts}]
Let $e_{1}=(1,0)$ and $e_{2}=(0,1)$ be the unit vectors. Let $c=c_{k, \kappa}>0$ be a small number that is to be chosen. For any $\delta>0$, the set
\[
\bigcap_{a_{1} \in \Set{0,\dotsc,k}} (\Gamma - c\delta a_{1}e_{1})
\]
will be a semialgebraic set of complexity at most $\kappa' = \kappa'(k,\kappa)$.
Using Corollary~\ref{cor:shift-of-bdd-complexity} in the first variable and integrating, we see that
\[
\abs[\big]{ \bigcap_{a_{1} \in \Set{0,\dotsc,k}} (\Gamma - c\delta a_{1}e_{1}) } \geq \delta - C_{\kappa} c \delta k,
\]
for some large constant $C_{\kappa}$ depending only on $\kappa$.  If $c \leq 1/(4kC_{\kappa})$, this will be $\geq 3\delta/4$. Using Corollary~\ref{cor:shift-of-bdd-complexity} in the second variable and integrating, we see that
\[
\abs[\big]{ \bigcap_{a \in \Set{0,\dotsc,k}^{2}} (\Gamma - c\delta a) }
=
\abs[\big]{ \bigcap_{a_{2} \in \Set{0,\dotsc,k}} \bigl( \bigcap_{a_{1} \in \Set{0,\dotsc,k}^{2}} (\Gamma - c\delta a_{1}e_{1}) \bigr) - c\delta a_{2}e_{2}}
\geq
3\delta/4 - C_{\kappa'} c \delta k,
\]
for some large constant $C_{\kappa'}$ depending only on $\kappa'$. If $c \leq 1/(4kC_{\kappa'})$, then this is $\geq \delta/2$.
Summarizing, it suffices to take $c_{k,\kappa} := 1/\max(4kC_{\kappa},4kC_{\kappa'})$.
\end{proof}

\begin{rem}
Lemma \ref{lem:large-intersection-of-shifts} is also true when the dimension $2$ is replaced by an arbitrary $d$. One just needs to repeat the procedure in the proof $d$ times. We do not need that more general result in this paper but record it here as it may be of independent interest.
\end{rem}

\subsection{Large projections of stationary sets}\label{subsection4_1}
In this subsection, we fix $x\in L_{\delta, \mathrm{sprd}}$, and discuss some geometry of a particular stationary set $\gammax$.

We apply Lemma~\ref{lem:large-intersection-of-shifts} to the set $\gammax$. Its complexity is bounded by a constant depending on $k$. Therefore, we can find a small constant $c_k>0$ such that for the set
\begin{equation}\label{210121e4_2}
\tilde{Z}_{x} :=
\bigcap_{a \in \Set{0,\dotsc,k}^{2}} \bigl( \gammax - c_{k}\delta a \bigr),
\end{equation}
we have $\abs{\tilde{Z}_{x}} > \delta/2$. If we also know that $x\in L_{\delta, \mathrm{sprd}}$, then we will see that the set $\tilde{Z}_x$ has a ``large'' projection in the horizontal axis.

\begin{lem}
\label{lem:many-strips}
For every $x\in L_{\delta, \mathrm{sprd}}$, it holds that
\begin{equation}
\abs{\Set{S_{K} \in \calS_{K} \given \tilde{Z}_{x} \cap S_{K} \neq \emptyset}} \geq C_{k,1}/2.
\end{equation}
\end{lem}
\begin{proof}[Proof of Lemma \ref{lem:many-strips}]
Let $\tilde{\calS}_{K} \subseteq \calS_{K}$ be the set consisting of those strips that have non-empty intersection with $\tilde{Z}_{x}$.
Using that $x \not\in L_{\delta}(S_{K})$ for any $S_{K} \in \calS_{K}$, we obtain
\[
\delta/2 < \abs{\tilde{Z}_{x}}
=
\sum_{S_{K} \in \tilde{\calS}_{K}} \abs{\tilde{Z}_{x} \cap S_{K}}
\leq
\sum_{S_{K} \in \tilde{\calS}_{K}} \abs{Z_{x} \cap S_{K}}
\leq
\sum_{S_{K} \in \tilde{\calS}_{K}} \delta/C_{k,1}
=
\delta \abs{\tilde{\calS}_{K}}/C_{k,1}.
\]
Rearranging the terms gives the claim.
\end{proof}

\begin{lemma}
\label{lem:large-projection}
If $C_{k,1}$ is large enough depending on $k$, then, for every $x\in L_{\delta, \mathrm{sprd}}$, the projection $\pi_{1}\tilde{Z}_{x}$ contains a dyadic interval of length $1/K$. Here $\pi_{1}: \R^2\to \R$ denote the projection to the first variable.
\end{lemma}
\begin{proof}[Proof of Lemma \ref{lem:large-projection}]
Using an effective version of quantifier-elimination in the theory of real closed fields
(that was mentioned in Subsection 2.1; see e.g. \cite[Theorem 14.16]{MR2248869} for a reference)
$\pi_{1} \tilde{Z}_{x}$ is semialgebraic set having bounded complexity depending only on $k$.
By Lemma~\ref{lem:finite-boundary}, it is the union of at most $\tilde{C}_{k}$ intervals, where $\tilde{C}_{k}$ depends only on $k$.
If $C_{k,1} > 4\tilde{C}_{k}$, then it follows from Lemma~\ref{lem:many-strips} that $\pi_{1}\tilde{Z}_{x}$ contains at least one full dyadic interval of length $1/K$.
\end{proof}

\begin{rem}
From this point on, all estimates are allowed to depend implicitly on both $k$ and $K$; these implicit constants will accumulate to the constant $C_{k, K}$ in \eqref{eq:Ldelta-recursion}.
To simplify notation, the dependence on $K$ and $k$ will often be dropped.
\end{rem}

By Lemma~\ref{lem:large-projection} and a standard pigeonholing argument, we can find a subset $\widetilde{L}_{\delta, \mathrm{sprd}}\subset L_{\delta, \mathrm{sprd}}$ and an interval $I \subset [0,1]$ of length $1/K$ such that
\[
|\widetilde{L}_{\delta, \mathrm{sprd}}| \geq |L_{\delta, \mathrm{sprd}}|/K
\]
and for every $x\in \widetilde{L}_{\delta, \mathrm{sprd}}$ we have $I \subseteq \pi_{1} \tilde{Z}_{x}$.
Let
\begin{equation}
t_0, t_1, \dots, t_k\in I
\end{equation}
be a $1/(kK)$-separated set.

Let $\Part{\delta}$ be the partition of $[0, 1]^2$ into dyadic squares $\Delta$ of side length $\delta$. Next, for each $x\in L_{\delta, \mathrm{sprd}}$, we will construct a sub-collection $\mc{P}_x(\delta)\subset \mc{P}(\delta)$.

Roughly speaking, the purpose of this construction is that we want $Z_x$ to be approximated by a union of certain lattice dyadic $\delta$-squares. To this end, we choose sufficiently many $\Delta \in \mc{P}_x(\delta)$ such that the union of all $\Delta \in \mc{P}_x(\delta)$ approximates $Z_x$ well. Moreover when $\Delta \in \mc{P}_x(\delta)$, we want to have $\abs{ P(\xi; x)- \mu_x} \lesssim 1$ for \emph{every} $\xi \in \Delta$ so that $\bigcup_{\Delta \in \mc{P}_x(\delta)} \Delta$ is in a slightly thickened stationary set. This is guaranteed by the following Lemma whose proof uses the Lagrange interpolation.
\begin{lem}\label{lem:lagrange-interpolation}
Let $\calP_{x}(\delta) := \Set{\Delta \in \Part{\delta} \given \Delta \cap \tilde{Z}_{x} \neq \emptyset}$.
Then, for each $\Delta \in \calP_{x}(\delta)$ and $\xi\in\Delta$, we have
\begin{equation}
\label{eq:P-almost-const-near-Z}
\abs{ P(\xi; x)- \mu_x} \lesssim_{k, K} 1.
\end{equation}
\end{lem}
\begin{proof}[Proof of Lemma \ref{lem:lagrange-interpolation}]
In the proof of this lemma we will see the motivation of introducing \eqref{210121e4_2}. We take one $\Delta\in \calP_x(\delta)$. By definition, we can find $\xi_0\in \Delta$ such that $\xi_0\in \tilde{Z}_x$. By the definition of $\tilde{Z}_x$ as in \eqref{210121e4_2}, we know that
\begin{equation}
\xi_0+c_k \delta a\in Z_x \text{ for every } a\in \{0, 1, \dots, k\}^2.
\end{equation}
As a consequence, we obtain that
\begin{equation}
\mu_x\le P(\xi_0+c_k \delta a; x)\le \mu_x+1.
\end{equation}
The desired bound \eqref{eq:P-almost-const-near-Z} follows from applying the interpolation polynomials in the Lagrange form.
\end{proof}

In particular,
\begin{equation}
t_{k'}\in \pi_1 \Delta \text{ for some } \Delta\in \mc{P}_{x}(\delta),
\end{equation}
for every $x\in \widetilde{L}_{\delta, \mathrm{sprd}}$ and every $0\le k'\le k$.

\subsection{Rigidity of stationary sets}\label{201118subsection7_3}

For $\Xi\subset \R^N$, let $\con(\Xi)$ be the closure of the convex hull of $\Xi$. By John's ellipsoid lemma, we know that there exists a rectangular box $\Box_{\Xi}$ such that
\begin{equation}
\Box_{\Xi}\subset \con(\Xi)\subset C_N \Box_{\Xi},
\end{equation}
where $C_N$ is a large constant that depends only on $N$. Let $c(\Xi)$ denote the center of $\Box_{\Xi}$.  Let $\dual(\Xi)$ denote
\begin{equation}
\{x\in \R^N: |\langle \bxi-c(\Xi), x\rangle|\le 2C_{k, K} \text{ for every } \bxi\in \Box_{\Xi}\},
\end{equation}
where $C_{k, K}$ is the same as the one in Lemma \ref{lem:lagrange-interpolation}. Notice that $|\Box_{\Xi}||\dual(\Xi)|\sim 1$. \\

Recall from \eqref{eq:P-almost-const-near-Z} that, for every $\xi\in \Delta\in \mc{P}_x(\delta)$, we have
\begin{equation}
\mu_x-C_{k, K}\le P(\xi; x)=x\cdot \phi_{ k}(\xi)\le \mu_x+C_{k, K}.
\end{equation}
Denote
\begin{equation}
\Xi_x:=\{\phi_{ k}(\xi): \xi\in \Delta\in \mc{P}_x(\delta)\}.
\end{equation}
We obtain that
\begin{equation}
\mu_x-C_{k, K}\le x\cdot \bxi\le \mu_x+C_{k, K},
\end{equation}
for every $\bxi\in \con(\Xi_{x})$, and therefore
\begin{equation}
|x\cdot (\bxi-c(\Xi_{x}))|\le 2C_{k, K},
\end{equation}
for every $\bxi\in \con(\Xi_{x})$, which further means $x\in \dual(\Xi_{x})$.

\subsection{Estimating dual boxes} In this section, we will bound $\widetilde{L}_{\delta, \mathrm{sprd}}$. Recall the choice of $t_0, t_1, \dots, t_k$ from  previous subsections. Let $\mc{P}_{k'}(\delta)$ denote
\begin{equation}
\{\Delta \in \mc{P}(\delta) : (t_{k'}, s)\in \Delta \text{ for some } s\in [0, 1]\}.
\end{equation}
For a fixed tuple
\begin{equation}
\bdelta=(\Delta_0, \dots, \Delta_{k}) \in \mc{P}_0(\delta)\times \dots \times \mc{P}_{k}(\delta),
\end{equation}
denote
\begin{equation}
\widetilde{L}_{\delta, \mathrm{sprd}}(\bdelta):=\{x\in \widetilde{L}_{\delta, \mathrm{sprd}}: \Delta_{k'}\in \mc{P}_x(\delta) \text{ for every } k'\}.
\end{equation}
Therefore we can write
\begin{equation}\label{201121e7_21}
\widetilde{L}_{\delta, \mathrm{sprd}}=\bigcup_{\bdelta \in \mc{P}_0(\delta)\times \dots \times \mc{P}_{k}(\delta)} \widetilde{L}_{\delta, \mathrm{sprd}}(\bdelta).
\end{equation}

Let us focus on estimating $\widetilde{L}_{\delta, \mathrm{sprd}}(\bdelta)$. Recall the discussion in Subsection \ref{201118subsection7_3}. Let us use $\phi_k(\bdelta)$ to denote $\{\phi_{k}(\xi): \xi\in \Delta_{k'} \text{ for some } k'\}$. We have
\begin{equation}
\label{eq:3}
|\widetilde{L}_{\delta, \mathrm{sprd}}(\bdelta)|
\le |\dual(\phi_k(\bdelta))|
\sim \abs{\con(\phi_k(\bdelta))}^{-1}
\end{equation}
We will get a good lower bound on $\abs{\con(\phi_k(\bdelta))}$ by choosing a transverse set of line segments inside this convex set.
A standard geometric observation is that each set $\con(\phi_{k}(\Delta))$ contains line segments of length $\sim \delta^{-l}$ in the direction of $l$-th derivatives of $\phi_{k}$:
\begin{lem}
\label{lem:tangent-lines}
For every $\Delta \in \Part{\delta}$, every $\xi\in\Delta$, and every multiindex $\beta$ with $1 \leq \abs{\beta} \leq k$, the convex hull $\con(\phi_{k}(\Delta))$ contains a line segment of length $\sim\delta^{\abs{\beta}}$ in the direction $\partial^{\beta}\phi_{k}(\xi)$.
\end{lem}

From now on, we order the monomials in $\phi$ in the increasing lexicographic order and write
\begin{equation}
\phi_{k}(\xi)=(t, t^2, \dots, t^k, s, st, \dots, st^{k-1}, s^2, \dots),
\quad \xi = (s,t).
\end{equation}
We also order the $t_{j}$'s in a convenient way.
By the Vandermonde determinant formula, we know that
\[
\det \begin{pmatrix}
1 & t_0, & t_0^2, & \dots & t_0^{l}\\
1 & t_1, & t_1^2, & \dots & t_1^{l}\\
\vdots & \vdots & \vdots & \ddots & \vdots\\
1 & t_{l}, & t_{l}^2, & \dots & t_{l}^{l}
\end{pmatrix}
\sim 1.
\]
Let $v_{l} \in \R^{l+1}$ be the vector whose Hodge dual is
\[
\star v_{l} = \begin{pmatrix}
1\\
1\\
\vdots\\
1
\end{pmatrix}
\wedge
\begin{pmatrix}
t_0\\
t_1\\
\vdots\\
t_{l}
\end{pmatrix}
\wedge \dotsb \wedge
\begin{pmatrix}
t_0^{l-1}\\
t_1^{l-1}\\
\vdots\\
t_{l}^{l-1}
\end{pmatrix}.
\]
It follows that $\norm{v_{l}} \sim 1$.
By rearranging $t_{0},\dotsc,t_{l}$ if necessary, we may assume that the last entry of $v_{l}$ is $\sim 1$.
Doing this in the reverse order of $l=k,\dotsc,1$, we may assume that this holds for every $l$.

With this order of $t_{j}$'s, we will use the following lower bound:
\begin{align*}
& \abs{\con(\phi_k(\bdelta))}
\\ & \geq
\abs[\Big]{ \bigwedge_{j=1}^{k} (\phi(\xi_{j})-\phi(\xi_{0}))
  \wedge \bigwedge_{j=0}^{k} \delta \partial_{s}\phi(\xi_{j})
  \wedge \bigwedge_{k'=2}^{k-1} \bigwedge_{j=1}^{k+1-k'} \delta^{k'} (\partial_{s}^{k'} \phi(\xi_{j}) - \partial_{s}^{k'} \phi(\xi_{0})) }
\\ &=
\delta^{k(k+1)(k+2)/6-k+1}
\abs[\Big]{\bigwedge_{j=1}^{k} (\phi(\xi_{j})-\phi(\xi_{0}))
  \wedge \bigwedge_{j=0}^{k} \partial_{s}\phi(\xi_{j})
  \wedge \bigwedge_{k'=2}^{k-1} \bigwedge_{j=1}^{k+1-k'} (\partial_{s}^{k'} \phi(\xi_{j}) - \partial_{s}^{k'} \phi(\xi_{0}))}.
\end{align*}
Here, we have chosen a total of
\[
k + (k+1) + (k-1) + \dotsb + 2
= k + k(k+1)/2 = N
\]
vectors in the difference set of the convex set $\con(\phi_{k}(\bdelta))$, of which the first $k$ come from differences between $\Delta_{0}$ and $\Delta_{1},\dotsc,\Delta_{k}$, while the remaining ones come from individual $\Delta_{j}$'s via Lemma~\ref{lem:tangent-lines}.

Writing these vectors as rows of a matrix, we see that, up to a $(k, l)$-dependent constant, the latter wedge product is the determinant of the block upper triangular matrix
\[
\begin{pmatrix}
A_{k} & * & \dots & \dots & * \\
0 & \tilde{A}_{k-1} & * & &\vdots \\
0 & 0 & \bar{A}_{k-2} & \ddots & \vdots \\
\vdots & & \ddots & \ddots &  * \\
0 & \dots & \dots & 0 & \bar{A}_1
\end{pmatrix}
\]
with
\[
A_{k}=\begin{bmatrix}
t_1-t_{0}, & t_1^2-t_{0}^{2}, & \dots & t_1^{k}-t_{0}^{k}\\
t_2-t_{0}, & t_2^2-t_{0}^{2}, & \dots & t_2^{k}-t_{0}^{k}\\
\vdots & \vdots & \ddots & \vdots\\
t_k-t_{0}, & t_k^2-t_{0}^{2}, & \dots & t_k^{k}-t_{0}^{k}
\end{bmatrix}
\]
\[
\tilde{A}_{l}:=\begin{bmatrix}
1 & t_0, & t_0^2, & \dots & t_0^{l} & s_{0}\\
1 & t_1, & t_1^2, & \dots & t_1^{l} & s_{1}\\
\vdots & \vdots & \vdots & \ddots & \vdots & \vdots\\
1 & t_{l+1}, & t_{l+1}^2, & \dots & t_{l+1}^{l} & s_{l+1}
\end{bmatrix}
\]
\[
\bar{A}_{l}:=
\begin{bmatrix}
t_1-t_{0}, & \dots & t_1^{l}-t_{0}^{l} & (s_{1}-s_{0})\\
\vdots &  \ddots & \vdots & \vdots\\
t_{l+1}-t_{0}, & \dots & t_{l+1}^{l}-t_{0}^{l} & (s_{l+1}-s_{0})
\end{bmatrix}
\]
Since $A_{k}$ is a Vandermonde matrix, its determinant is $\sim 1$.
Consider next
\begin{equation}
\label{eq:2}
\det \bar{A}_{l} = \det \tilde{A}_{l} = v_{l+1} \cdot (s_{0},\dotsc,s_{l+1}).
\end{equation}
We fix some $s_{0},s_{1}$ with
\begin{equation}
\label{eq:1}
(t_{i},s_{i}) \in \Delta_{i}
\end{equation}
and select $s_{l+1}$ for $l=1,\dotsc,k-1$ recursively in such a way that \eqref{eq:1} holds and the absolute value of \eqref{eq:2} is maximized.
By our arrangement of $t_{j}$'s, the last entry of $v_{l+1}$ is $\sim 1$.
Therefore, given $s_{0},\dotsc,s_{l}$, the set of $s_{l+1}$'s for which $\abs{\det\bar{A}_{l}} \sim \delta_{l}$ is the union of at most two intervals of length $\sim \delta_{l}$.
It follows that, given $\Delta_{0},\dotsc,\Delta_{l}$ and $\delta_{l} \in [\delta,1]$, there are $\sim \delta_{l}/\delta$ many $\Delta_{l+1}$ for which $\abs{\det \bar{A}_{l}} \sim \delta_{l}$, and there are no other possibilities.

Therefore,
\begin{align*}
& \sum_{\bdelta\in \mc{P}_0(\delta)\times \dots \times \mc{P}_{k}(\delta)} \abs{\con(\phi_k(\bdelta))}^{-1}
\\ &\leq
\sum_{\substack{\delta_1, \dots, \delta_{k-1} \in [\delta,1]\\ \text{dyadic}}} \sum_{\Delta_{0},\Delta_{1}} \sum_{\substack{\Delta_{2} \\ \abs{\det \bar{A}_{2}}\sim \delta_{1}}} \dots \sum_{\substack{\Delta_{k} \\ \abs{\det \bar{A}_{k}}\sim \delta_{k-1}}} \abs{\con(\phi_k(\bdelta))}^{-1}
\\ &\lesssim
(\delta^{k(k+1)(k+2)/6-k+1})^{-1} \sum_{\substack{\delta_1, \dots, \delta_{k-1} \in [\delta,1]\\ \text{dyadic}}} \sum_{\Delta_{0},\Delta_{1}} \sum_{\substack{\Delta_{2} \\ \abs{\det \bar{A}_{1}}\sim \delta_{1}}} \dots \sum_{\substack{\Delta_{k} \\ \abs{\det \bar{A}_{k-1}}\sim \delta_{k-1}}} \bigl( \det A_{k} \prod_{j=1}^{k-1} \det \bar{A}_{j}\bigr)^{-1}
\\ &\lesssim
(\delta^{k(k+1)(k+2)/6-k+1})^{-1} \sum_{\substack{\delta_1, \dots, \delta_{k-1} \in [\delta,1]\\ \text{dyadic}}} \sum_{\Delta_{0},\Delta_{1}} \sum_{\substack{\Delta_{2} \\ \abs{\det \bar{A}_{1}}\sim \delta_{1}}} \dots \sum_{\substack{\Delta_{k} \\ \abs{\det \bar{A}_{k-1}}\sim \delta_{k-1}}} \bigl( \prod_{j=1}^{k-1} \delta_{j} \bigr)^{-1}
\\ &\lesssim
(\delta^{k(k+1)(k+2)/6-k+1})^{-1} \sum_{\substack{\delta_1, \dots, \delta_{k-1} \in [\delta,1]\\ \text{dyadic}}} \delta^{-k-1}
\\ &\lesssim
\abs{\log_{+} \delta}^{k-1}
\delta^{-k(k+1)(k+2)/6-2}.
\end{align*}
This finishes the estimate for \eqref{201121e7_21} and the proof of Theorem~\ref{main_tarry}.

\begin{rem}\label{rem:higher-dim-tarry}
If we follow the proof of Theorem~\ref{main_tarry} in higher dimensions $d\geq 3$, the problem of estimating the spread out set becomes much more difficult.
In dimension $d=2$, we greatly benefited from a factorization of the upper bound \eqref{eq:3} into linear factors given by \eqref{eq:2}.
Unless a similar upper bound with good factorization can be found in higher dimensions, we would face complicated quantitative real algebraic problems.
We feel it may be difficult to find such an upper bound at this moment, and hence do not pursue the problem for $S_{d, k}$, $d\geq 3$, in the current paper.
\end{rem}

\section{Lower bounds}\label{210216section5}

In this section, we prove lower bounds for $p_{2, k}$. The idea is not difficult to explain. We will show that the function $E_k 1$ is ``large" (see Lemma \ref{210104lem6_2}) on a disjoint union of rectangular boxes (see Lemma \ref{210104lem6_1}). This idea has already been used by  Arkhipov, Chubarikov and Karatsuba \cite{MR552548, MR2113479} and by Ikromov \cite{MR1636721}. The lower bound in Lemma \ref{210104lem6_2} is essentially the same as in the above mentioned papers; the key difference is that we are able to find more rectangular boxes where such a lower bounds holds. Roughly speaking, this is achieved via applying the ``rotation symmetry" of Parsell-Vinogradov surfaces. \\

Let $\lambda>0$ be a large number. Let $\theta_r=r/(100\lambda)$, $r=0, 1, \dots, \lambda$. Let $w_r=1/4+r/(2\lambda)$, $r=0, 1, \dots, \lambda$. The choice of $\theta_r$ and $w_r$ is such that
\begin{equation}
|\theta_r|\le \frac{1}{100}, \ \ \frac{1}{4}\le w_r\le \frac{3}{4},
\end{equation}
and what we will need is that $\theta_r$ is close to 0, and $w_r$ is away from both 0 and 1.

Recall that $x\in \R^N$. Let us write
\begin{equation}\label{210121e5_2}
P(\xi; x)=x_1 \xi_1^k+ \cdots.
\end{equation}
For given $\theta_{r}$ and $w_{r'}$, we let $\bfr=(r, r')$ and write
\begin{equation}\label{210206e5_3}
P(\xi; x)=\sum_{0\le k'\le k} \gamma_{\bfr, k'}(\xi_2; x)(\xi_1+\theta_r \xi_2-w_{r'})^{k'},
\end{equation}
where for each $\bfr$, $k'$ and $x\in \R^N$, the function $\gamma_{\bfr, k'}(\xi_2; x)$ is a polynomial of degree $k-k'$ in $\xi_2$. In particular, $\gamma_{\bfr, k}(\xi_2; x)$ is a constant function that is independent of $\bfr$ and $\xi_2$; indeed, $\gamma_{\bfr, k}(\xi_2; x)=x_1$, which follows from the normalization in \eqref{210121e5_2}.  Define
\begin{multline}\label{210206e5_4}
\Omega_{\bfr}:=\{x\in \R^N: \lambda^k\le x_1\le (2\lambda)^k, \|\gamma_{\bfr, k'}(\cdot; x)\|\le \epsilon \lambda^{k'}, k'=1, 2, \dots, k-1,\\
\|\gamma_{\bfr, 0}(\cdot; x)-\gamma_{\bfr, 0}(0; x)\|\le\epsilon \}.
\end{multline}
Here $\epsilon>0$ is a sufficiently small parameter that will be picked later, and for a polynomial $\gamma: [0, 1]\to \R$, we use $\|\gamma\|$ to denote the $\ell^1$ sum of all its coefficients.
\begin{lem}\label{210104lem6_1}
If $\epsilon$ (depending only on $k$) is chosen to be sufficiently small, then for each $\bfr$, it holds that
\begin{equation}\label{210104e6_5}
|\Omega_{\bfr}|\gtrsim_k \lambda^{k(k+1)(k+2)/6},
\end{equation}
and for $\bfr_1\neq \bfr_2$, we have $\Omega_{\bfr_1}\cap \Omega_{\bfr_2}=\emptyset$.
\end{lem}
\begin{lem}\label{210104lem6_2}
For each $\bfr$ and each $x\in \Omega_{\bfr}$, it holds that
\begin{equation} \label{eq:E1-lower-bd}
\Big|\int_{[0, 1]^2} e(P(\xi; x))d\xi\Big|\gtrsim \lambda^{-1},
\end{equation}
provided that $\epsilon$ is chosen to be sufficiently small, depending only on $k$.
\end{lem}
The desired lower bound follows immediately from these two lemmas.
\begin{rem}\label{210216rem}
Lemma \ref{210104lem6_1} and Lemma \ref{210104lem6_2} also immediately lead to the sharpness of Corollary \ref{coro1_3}. Indeed, they imply that the bound \eqref{201219e1_11} also fails whenever $p\le p_{d, k}$; in particular, the endpoint case $p=p_{d, k}$ is also included.
\end{rem}

In the end, we will prove Lemma \ref{210104lem6_1} and Lemma \ref{210104lem6_2}. We will start with the latter one. Its proof relies on the following lemma in Ikromov \cite{MR1636721}, a slightly different form of which also already appeared in Arkhipov, Chubarikov and Karatsuba \cite{MR552548, MR2113479}.
\begin{lem}[\cite{MR1636721}]\label{210121lem5_3}
Let $a<-10^{-2}$ and $b>10^{-2}$. Suppose that the polynomial $q(x)=\alpha_k x^k+\dots+ \alpha_1 x$ satisfies
\begin{equation}
A^k\le \alpha_k\le (2A)^k, \ |\alpha_{k-1}|\le \epsilon A^{k-1}, \dots, |\alpha_1|\le \epsilon A,
\end{equation}
for some large $A$ and small $\epsilon$. Then we have the following asymptotic representation
\begin{equation}
\int_a^b e^{2\pi i q(x)}dx=\frac{c(\alpha)}{\alpha_k^{1/k}}+ O(A^{1-k})  \text{ as } A\to \infty,
\end{equation}
where $c(\alpha)$ is a constant depending on $\alpha_k , \dotsc, \alpha_1$ satisfying
\begin{equation}
c(\alpha)=c+o(\epsilon) \text{ as } \epsilon\to 0,
\end{equation}
with $c\neq 0$ depending only on $k$.
\end{lem}
\begin{proof}[Proof of Lemma \ref{210104lem6_2}]
Let us write our integral as
\begin{equation}
\int_{[0, 1]^2} e\Big(\sum_{0\le k'\le k} \gamma_{\bfr, k'}(\xi_2; x)(\xi_1+\theta_r \xi_2-w_{r'})^{k'}\Big)d\xi.
\end{equation}
After the change of variable
\begin{equation}
\xi_1+\theta_r\xi_2-w_{r'}\mapsto \tilde{\xi}_1
\end{equation}
for each fixed $\xi_{2}$, the integral becomes
\begin{equation}
e(\gamma_{\bfr, 0}(0; x)) \int_0^1 e(\gamma_{\bfr, 0}(\xi_2; x)-\gamma_{\bfr, 0}(0; x)) \int_{\theta_r\xi_2-w_{r'}}^{1+\theta_r \xi_2-w_{r'}} e\Big(\sum_{0< k'\le k} \gamma_{\bfr, k'}(\xi_2; x) \tilde{\xi}_1^{k'}\Big)d\tilde{\xi}_1 d\xi_2.
\end{equation}
In the case $k\geq 3$, since $\gamma_{\bfr, k}(\xi_2; x) = x_{1}$, Lemma~\ref{210121lem5_3} gives a uniform lower bound for the inner integral, provided that $\epsilon$ is chosen to be sufficiently small.
In the case $k=2$, we complete squares in the phase and use the fact that the Fresnel integral $\int_{0}^{u} \cos(\xi^{2}) d\xi$ is bounded below by a positive constant as soon as $u>0$ is bounded away from $0$, while the Fresnel integral $\int_{0}^{u} \sin(\xi^{2}) d\xi$ is bounded above.

In both cases, the value of the inner integral is contained in a cone in $\C$ with angle $<\pi$ and bounded away from $0$.
Using the upper bound $\abs{\gamma_{\bfr, 0}(\xi_2; x)-\gamma_{\bfr, 0}(0; x)} \lesssim \epsilon$, we can complete the proof of \eqref{eq:E1-lower-bd}.
\end{proof}
\begin{proof}[Proof of Lemma \ref{210104lem6_1}]
Regarding the (elementary) lower bound \eqref{210104e6_5}, we first mention that this has also been observed by Ikromov in the case $\theta_r=0$, see \cite[page 183]{MR1636721}.
To show \eqref{210104e6_5}, note that the coefficient of $\xi_{2}^{l}$ in $\gamma_{\bfr, k'}(\xi_{2},x)$ equals $x_{k',l}$ plus some function of $\theta,w$ and the components $x_{k'',l}$ with $k' < k'' \leq k$.

Hence, for each $k'=k,\dotsc,0$, once $x_{k'',l}$ have been chosen for all $k''>k'$ and all $l$, we can choose each $x_{k',l}$ with $l \in \Set{0,\dotsc,k-k'}$ in an interval of length $\sim \lambda^{k'}$.
The measure of the set of all $x$ that we can choose is then proportional to $\lambda$ to the power
\[
\sum_{k'=0}^{k} k'(k-k'+1) = \frac{1}{6} k(k+1)(k+2).
\]

In the end, we show that $\Omega_{\bfr_1}\cap \Omega_{\bfr_2}=\emptyset$ whenever $\bfr_1\neq \bfr_2$. We argue by contradiction. Let $\bfr_1\neq \bfr_2$ and assume that there exists an $x\in \Omega_{\bfr_1}\cap \Omega_{\bfr_2}$. We now expand $P(\xi; x)$ in the form of \eqref{210206e5_3} with $(\theta_{r_1}, w_{r'_1})$ and $(\theta_{r_2}, w_{r'_2})$ separately. By only considering the first term in these two expansions, we obtain
\begin{equation}
x_1(\xi_1+\theta_{r_1}\xi_2-w_{r'_1})^k+\dots\equiv    x_1(\xi_1+\theta_{r_2}\xi_2-w_{r'_2})^k+\dots
\end{equation}
as functions depending on $\xi_1$ and $\xi_2$. Recall the definition \eqref{210206e5_4}. We will arrive at a contradiction if we can show that the following polynomial in $\xi_1$ and $\xi_2$
\begin{equation}\label{210207e5_15}
x_1(\xi_1+\theta_{r_1}\xi_2-w_{r'_1})^k-    x_1(\xi_1+\theta_{r_2}\xi_2-w_{r'_2})^k
\end{equation}
has a coefficient of a non-constant term with absolute value $\gtrsim \lambda^{k-1}$.

The monomial $\xi_1^{k-1}\xi_2$ has the coefficient $k x_1(\theta_{r_1}-\theta_{r_2})$, and the monomial $\xi_1^{k-1}$ has the coefficient $k x_1(w_{r_1'}-w_{r_2'})$.
At least one of these coefficients has absolute value $\gtrsim |x_{1}| / \lambda \sim \lambda^{k-1}$.
\end{proof}

\section{Proof of Corollary \ref{coro1_3}}\label{sectionrestr}

Fix $d=2$, $k\ge 2$ and $f$ with $\|f\|_{\infty}=1$. We abbreviate $\bar{p}_{2, k}$ to $\bar{p}_k$. Let $\mu$ be the measure supported on the surface $S_{2, k}$ whose projection to $\R^2$ is given by $f(\xi)d\xi$, and let $\mu_0$ be that given by  $\one_{[0, 1]^2}(\xi)d\xi$. We can therefore write
\begin{equation}
E^{(2, k)}_{[0, 1]^2}f=\widehat{\mu} \text{ and } E^{(2, k)}_{[0, 1]^2}1=\widehat{\mu_0}.
\end{equation}
Let $\chi: \R^N\to \R$ be a non-negative smooth bump function supported on $[-1, 1]^N$ satisfying that $\widehat{\chi}$ is non-negative and  $\widehat{\chi}(x)\gtrsim_N 1$ for every $|x|\le 1$.  Define $\chi_R(\xi):=R^N\chi(R\xi)$.
Recall that $\bar{p}_{k}$ is the smallest even number $\ge p_{2, k}$.  As a consequence of Theorem \ref{main_tarry} and H\"older's inequality, we obtain that for every $\epsilon>0$ and $R\ge 1$, it holds that
\begin{equation}
\|\widehat{\mu_0}\|_{L^{\bar{p}_{k}}(B_R)}\lesim_{\epsilon} R^{\epsilon},
\end{equation}
where $B_R\subset \R^N$ is an arbitrary ball of radius $R$. We will show that
\begin{equation}\label{210207e6_3}
\|\widehat{\mu}\|_{L^{\bar{p}_{k}}(B_R)}\lesim_{\epsilon} R^{\epsilon}.
\end{equation}
To prove this estimate, we first observe that, by modulating the function $f$, which does not change its $L^{\infty}$ norm, it suffices to consider the case that $B_R$ is centered at the origin. In this case, we have
\begin{equation}
\|\widehat{\mu}\|_{L^{\bar{p}_{k}}(B_R)}^{\bar{p}_k/2}\lesim \|\widehat{\mu}\cdot \widehat{\chi_R}\|_{L^{\bar{p}_{k}}(\R^N)}^{\bar{p}_k/2}= \big\|(\widehat{\mu}\cdot \widehat{\chi_R})^{\bar{p}_k/2}\big\|_2.
\end{equation}
We apply Plancherel's theorem to the last term and obtain
\begin{equation}
= \big\|(\mu*\chi_R)*\dots* (\mu*\chi_R)\big\|_2\le \big\|(\mu_0*\chi_R)*\dots* (\mu_0*\chi_R)\big\|_2.
\end{equation}
We apply Plancherel's theorem back and obtain \eqref{210207e6_3}. In the end, to pass from the local estimate \eqref{210207e6_3} to the desired global estimate in the corollary, one can follow Tao's epsilon-removal lemma in \cite{MR1666558} and its variants in Bourgain and Guth  \cite{MR2860188} and Kim  \cite{Kim2017SomeRO}.


\bibliographystyle{alpha}
\bibliography{reference.bib}

\end{document}